\documentclass[11pt,reqno,a4paper]{amsart}
\usepackage[latin1]{inputenc}
\usepackage[latin1]{inputenc}
\usepackage{a4wide}
\usepackage{fancyhdr}
\usepackage{amsfonts,amssymb,amsmath,amsthm,latexsym,indentfirst}
\usepackage{epsfig}
\usepackage{enumerate}
\usepackage{amsfonts}
\usepackage{latexsym}
\usepackage{amssymb}
\usepackage{amsmath}
\usepackage{textcomp}
\usepackage{graphicx}
\usepackage[bookmarks=false,colorlinks=true,linkcolor={blue},pdfstartview={XYZ null null 1.22}]{hyperref}
\usepackage{hyperref}

\numberwithin{equation}{section}

\newtheorem{theorem}{Theorem}[section]
\newtheorem{definition}[theorem]{Definition}
\newtheorem{lemma}[theorem]{Lemma}

\newtheorem{remark}[theorem]{Remark}
\newtheorem{corollary}[theorem]{Corollary}

\newcommand{\C}[1][]{\ensuremath{{\mathbb{C}^{#1}} }}
\newcommand{\R}[1][]{\ensuremath{{\mathbb{R}^{#1}} }}

\newcommand{\re}{\mathbb{R}}

\newcommand{\trinorm}{|\!|\!|}

\title[On the two-power nonlinear Schr\"odinger equation]{On the two-power nonlinear Schr\"odinger equation with non-local terms in Sobolev-Lorentz spaces}

\author[V. Barros, L.C.F. Ferreira, A. Pastor]{Vanessa Barros, Lucas C. F. Ferreira and Ademir Pastor}

\address{Universidade Federal da Bahia, Instituto de matem\'atica, Av. Adhemar de Barros , Ondina, 40170-110, Salvador, Bahia,
 Brazil}
\email{vbarrosoliveira@gmail.com}

\address{IMECC-UNICAMP, Rua S\'egio Buarque de Holanda, 651, Cidade UniversitÃ¡ria, 13083-859, Campinas, SÃ£o Paulo, Brazil}
\email{lcff@ime.unicamp.br}

\address{IMECC-UNICAMP, Rua S\'ergio Buarque de Holanda, 651, Cidade UniversitÃ¡ria, 13083-859, Campinas, SÃ£o Paulo, Brazil}
\email{apastor@ime.unicamp.br}



\begin{document}

\maketitle

\begin{abstract}
We are concerned with the two-power nonlinear Schr\"odinger-type equations with non-local terms. We consider the framework of Sobolev-Lorentz spaces which contain singular functions with infinite-energy. Our results include global existence, scattering and decay properties in this singular setting with fractional regularity index. Solutions can be physically realized because they have finite local $L^2$-mass. Moreover, we analyze the asymptotic stability of solutions and, although the equation has no scaling, show the existence of a class of solutions asymptotically self-similar w.r.t. the scaling of the single-power NLS-equation. Our results extend and complement those of [F. Weissler, ADE 2001], particularly because we are working in the larger setting of Sobolev-weak-$L^p$ spaces and considering non-local terms. The two nonlinearities of power-type and the generality of the non-local terms allow us to cover in a unified way a large number of dispersive equations and systems.
\end{abstract}

\section{Introduction}
We consider the initial value problem (IVP) associated with the two-power nonlinear
Schr\"odinger equation with non-local term
\begin{equation}\label{NLS}
\left\{\begin{array}{l} i \partial_t u + Lu  = a |u|^{\alpha}
 u +bE(|u|^\gamma) u,  \\
u(x,0)  = u_0(x)
\end{array}
\right.\quad  (x, t) \in \re^n\times \re,\ n\geq1,
\end{equation}
where $u=u(x,t)$ is a complex-valued function, $a$ and $b$ are complex
constants, $0<\alpha<\gamma$ are  positive real
numbers, $E$ is a non-local linear operator, and $L$ is a linear operator defined through its Fourier transform as
$$
\widehat{Lu}(\xi)=q(\xi)\widehat{u}, \quad \xi\in \R^n.
$$
Throughout the paper we assume the following:\\

{\bf (H1)} the function $q$ is real and homogeneous of degree $d$, that is,
$$
q(\lambda\xi)=\lambda^dq(\xi), \qquad \lambda>0.
$$

{\bf (H2)} The function $G(x)=\int_{\re^n} e^{i(x\xi+q(\xi))}d\xi$ belongs to
$L^\infty(\re^n)$.\\

{\bf (H3)} The operator $E$ is bounded in $L^{(p,\infty)}(\mathbb{R}^{n})$, for $1<p<\infty$, and commutes with fractional derivatives.\\

In \textbf{(H3)}, $L^{(p,\infty )}(\mathbb{R}^{n})$ stands for the so-called
weak-$L^{p}$ spaces which, in view of Chebyshev's inequality, can be
regarded as natural extensions of $L^{p}$-spaces. As will be seen later,
conditions \textbf{(H1)}-\textbf{(H2)} are sufficient to prove
dispersive-type estimates, which in turn are used to deal with the linear
part of the associated integral equation. Also, \textbf{(H3)} is sufficient
to handle with the nonlinear terms involving the operator $E$ in our functional setting.

When $L$ stands for the Laplacian operator and $b=0$,  the equation in \eqref{NLS} reduces to the well-known single-power nonlinear Schr\"odinger equation
\begin{equation}\label{singlenls}
i \partial_t u + \Delta u  = a |u|^{\alpha}u,
\end{equation}
which appears in many physical situations. So, at a first glance, \eqref{NLS} can be seen as a mathematical extension of \eqref{singlenls}. However,  several physical relevant models can be written in the form \eqref{NLS}. Indeed, let us recall some examples. When $n=2$, $L=m_1\partial_{x_1}^2+\partial_{x_2}^2$, $m_1\in\mathbb{R}\setminus\{0\}$, and $E$ is defined in Fourier variables as
\begin{equation}\label{E1}
\widehat{E(f)}(\xi)=\frac{\xi_1^2}{\xi_1^2+m_2\xi_2^2}\widehat{f}(\xi), \quad \xi=(\xi_1,\xi_2), \quad m_2>0,
\end{equation}
\eqref{NLS} reduces to the so-called Davey-Stewartson (DS) system, which was derived in \cite{DS} (see also \cite{GS}) to  model the evolution of weakly nonlinear water waves that travel predominantly in one direction, but in which the wave amplitude is modulated slowly in two horizontal directions. A generalized DS system describing the interaction of
longwaves and  shortwaves propagating in an infinite elastic medium (see \cite{BaE}) can also be written in the form \eqref{NLS}, where again $L=m_1\partial_{x_1}^2+\partial_{x_2}^2$, $m_1\in\mathbb{R}\setminus\{0\}$, and $E$ is defined by
\begin{equation}\label{E2}
\widehat{E(f)}(\xi)=\frac{\lambda\xi_1^4+(1+m_2-2\ell)\xi_1^2\xi_2^2+m_3\xi_2^4}{(\xi_1^2+m_2\xi_2^2)(\lambda\xi_1^2+m_3\xi_2^4)}\widehat{f}(\xi),
\end{equation}
with the real constants $m_2,m_3,\lambda$, and $\ell$ satisfying the relation
$$
(\lambda-1)(m_3-m_2)=\ell^2.
$$
An example in dimension $n=3$ is given by the Shrira system (see \cite{Sh}), which models the evolution of a three-dimensional packet of weakly
nonlinear internal gravity waves propagating obliquely at an arbitrary angle to the
vertical. In this case,
$$
L=\frac{\omega_{kk}}{2}\partial_{x_1}^2+\frac{\omega_{ll}}{2}\partial_{x_2}^2+\frac{\omega_{nn}}{2}\partial_{x_3}^2+\omega_{nk}\partial_{x_1x_3}^2
$$
with $\omega_{kk}\neq0$, $\omega_{nn}\neq0$,  $\omega_{ll}(\omega_{kk}\omega_{nn}-\omega_{nk}^2)\neq0$, and $E$ is given by
\begin{equation}\label{E3}
\widehat{E(f)}(\xi)=\nu\frac{\xi_2^2}{\xi_1^2+\xi_2^2}\widehat{f}(\xi), \quad \xi=(\xi_1,\xi_2,\xi_3), \quad \nu\in\mathbb{R}.
\end{equation}
Another three-dimensional example appears in the description of Bose-Einstein condensation of dipolar quantum gases (see \cite{LMSLP}, \cite{AS}). In such a case, $L$ is the Laplacian operator and $E$ is given by
\begin{equation}\label{E4}
\widehat{E(f)}(\xi)=\frac{4\pi}{3}\frac{2\xi_3^2-\xi_1^2-\xi_2^2}{|\xi|^2}\widehat{f}(\xi), \quad \xi=(\xi_1,\xi_2,\xi_3).
\end{equation}

For additional information concerning the above models and generalizations, we refer the reader to \cite{BEE}, \cite{BaE}, \cite{Ci}, \cite{DS}, \cite{G}, \cite{GS}, \cite{GS1}, \cite{GSh}, \cite{H1}, \cite{H2}, \cite{HH1}, \cite{HH2}, \cite{HS},  \cite{Sh}, \cite{W}). Note, in particular, that the symbols of the operators $E$ in \eqref{E1}, \eqref{E2}, and \eqref{E4} are homogeneous of degree zero. Hence, one can apply the Calderon-Zygmund theory (see, for instance, \cite{Gr}) to see that $E$ is bounded from $L^q$ to $L^q$, for any $1<q<\infty$; so, after applying real interpolation we see that {\bf (H3)} holds. For $E$ defined as in \eqref{E3} we cannot directly apply the Calderon-Zygmund theory. However, by interpolation with \textit{BMO} and Hardy spaces  we can still prove that $E$ is bounded from $L^q$ to $L^q$, for any $1<q<\infty$ (see, \cite[page 184]{GS1}), implying that {\bf (H3)} also holds in this case. In addition, recalling that
$$
\int_{\mathbb{R}}e^{i(a\eta^2+y\eta)}\,d\eta=\sqrt{\frac{\pi}{|a|}}e^{-i\frac{y^2}{4a}}e^{i\frac{\pi}{4}sgn(a)}, \qquad a,y\in\mathbb{R},
$$
it is easily seen that {\bf (H2)} also holds in the above examples.

Notice also that we are able to consider higher-order operators $L$, as we take the parameter $d\geq2$, such as polyharmonic operators. In particular, our results also apply to the fourth-order nonlinear Schrödinger equation
$$
i\partial_tu+\mu \Delta^2 u=a|u|^\alpha u,\qquad  \mu\neq0,
$$
which was introduced in \cite{Ka}, \cite{KS} to take into account the role of fourth-order dispersion in the
propagation of intense laser beams in a bulk medium with Kerr nonlinearity. To see that {\bf (H2)} also holds here, we refer the reader to \cite{BKS}.

Equation \eqref{singlenls} is one of most studied dispersive equation. To avoid too many references we cite, for instance, the books \cite{Ca}, \cite{Fi}, \cite{LP}, where the reader will find a large class of results in different function spaces. We only emphasize that results on self-similar solutions, among others, were obtained, for instance, in \cite{CW1}, \cite{CW2}, \cite{CW3}, \cite{RY1}, \cite{SFR}. Especially when $L$ is the Laplacian and $E$ is the identity operator $I$, \eqref{NLS} has been studied in recent years, to cite a few works, see \cite{CMZ}, \cite{TVZ}, \cite{MXZ-1}, \cite{MXZ-2} and their references. The most disseminated  results are obtained in the $H^1$-setting, which provides finite energy solutions. The studied topics cover local and global well-posedness, scattering, radial symmetry and asymptotic behavior of solutions. However,  in \cite{W}  solutions of \eqref{NLS} with $E=I$ were obtained in the spaces $H^{s,q}(\R^n)=(I-\Delta)^{-s/2}L^q(\R^n)$, with $s>0$ and $1<q<\infty$, provided that the initial data satisfies some additional conditions.

For a general operator satisfying the assumption {\bf (H3)},  existence and asymptotic properties in the single-power case $\gamma=\alpha$  were analyzed, in the context of weak Lebesgue spaces, in \cite{BP}, \cite{B} (see also \cite{VP}).

A substantial mathematical difference between \eqref{NLS} and \eqref{singlenls} is the failure of scaling in  \eqref{NLS}; it is easily seen that the map
\begin{equation}\label{scal-1}
u(x,t)\mapsto u_{\lambda}(x,t):=\lambda^{2/\alpha}u\left(\lambda x, \lambda^2 t \right),\quad \lambda>0
\end{equation}
leaves \eqref{singlenls} invariant (i.e., if $u$ is solution so is $u_{\lambda}$) but the same is not true for \eqref{NLS} (with $L=\Delta$ and $d=2$).

Our main aim here is twofold: to provide a larger class for global existence and uniqueness by considering the framework of Sobolev-weak-$L^p$ spaces with fractional regularity index $0<s<1$ (see Section \ref{preliminaries} for the definitions), which contain singular functions with infinite-energy and allow to analyze self-similar asymptotics; and to consider non-local operators in the nonlinearities which allow to address in a unified way a number of  dispersive models, including the above examples.

The aspects above, in turn, bring several additional difficulties. Thus we need to extend some well known results in the context of Lebesgue spaces to the spaces we are interested in, as well as to take into account the influence of the nonlocal operators in the estimates. We believe that those extended results have prospects to be applied in many other situations.

Also, we obtain scattering, decay properties of solutions and asymptotic stability results in that singular setting. In Remark \ref{rem2}, we point out that the data-solution map is Lipschitz continuous and then we have in fact a well-posedness result in the sense of Hadamard. Since weak-$L^p$ spaces embed into $L^2_{loc}$ for $p>2$, solutions have finite local $L^2$-mass and can be realized in physical space in any region with finite measure. In particular, all convergence and stability properties obtained here occur in the sense of $L^2$-mass in any finite-volume region, no matter how large it is.

Asymptotic stability results for NLS type equations and systems are useful
for characterizing solutions that, after initial perturbations, essentially
recovers their profiles at large times. Asymptotic (or not) self-similarity
type symmetries appear in physical situations and are used to describe
phenomena in different spatial-temporal scales, revealing internal symmetry
and structure in a system (see \cite{Du}, \cite{FV-R}). For instance, in
\cite{FKTD}, they show that a type of self-similar parabolic pulse, called
similaritons, is an asymptotic solution to a NLS-like equation with gain.
Although the equation \eqref{NLS} has no scaling, we are able to prove the
existence of a class of solutions asymptotically self-similar with respect
to the scaling \eqref{scal-1} of the  equation
$i\partial _{t}v+Lv=a|v|^{\alpha }v$, as $t\rightarrow \infty$. This means that, for large times, those solutions approximately preserve the self-similar structure of an important related model. Our self-similar asymptotics correspond to homogeneous initial data $u_{0}$ of degree $-d/\alpha $ and can be expressed as
\begin{equation}
v(x,t)=t^{-1/\alpha}V(xt^{-1/d}),  \label{self}
\end{equation}
where the profiles $V$ belong to some weak-$L^{p}$ spaces. This provides another motivation in order to consider weak-$L^{p}$ and Sobolev-weak-$L^{p}$ spaces since they are the natural environment for homogeneous functions and allow the analysis of self-similar asymptotics and pulses as in \cite{Du}, \cite{FV-R} and \cite{FKTD}. Moreover, the existence of self-similar asymptotics can also be used to study wave collapse (blow up solutions) by using the pseudo-conformal transformation (see \cite{CW1}). This singular behavior has appeared in the context of nonlinear optics and been observed in numerical experiments for some Schr\"odinger-type equations and
systems (see, e.g., \cite{Ku}). We believe our solutions may be obtained from numerical methods because they are the limit of a Picard sequence come from a contraction argument and are $L_{loc}^{2}$-stable.

As usual, to study the IVP  \eqref{NLS} we use its equivalent integral  formulation
\begin{equation} \label{eqintds}
u(t) = U(t) u_0 +i\int_{0}^{t}U(t-s)(a|u|^{\alpha} u +b E(|u|^{\gamma})u)(s)ds,
\end{equation}
where $U(t)u_0$ is the solution of the linear problem
\begin{equation*}\label{NLSlinear}
\left\{\begin{array}{l} i \partial_t u + Lu  = 0, \\
u(x,0)  = u_0(x)
\end{array}
\right.\quad  (x, t) \in \re^n\times \re,
\end{equation*}
that is,
\begin{equation*}\label{linear solution}
U(t)u_0(x)=\int_{\re^n}e^{i(x\xi+tq(\xi))}\widehat{u}_0(\xi)d\xi.
\end{equation*}
In view of our assumptions {\bf (H1)} and {\bf (H2)}  the operator $U(t)$ defines a unitary group on
$H^s(\re^n)$, for any $s\in \re$ (see \cite{GS1}). Without loss of generality, from now on we shall consider $t\geq0$ in \eqref{eqintds} and in time-dependent functional spaces dealt with. The case $t\leq0$ can be treated in a complete parallel way.

To simplify the calculations let us rewrite the integral equation \eqref{eqintds} as
\begin{equation*} \label{eqintds1}
u(t) = U(t) u_0 +B(u),
\end{equation*}
where
\begin{equation}\label{defB1}
 B(u)=i\int_{0}^{t}U(t-s)(a|u|^{\alpha} u +b E(|u|^{\gamma})u)(s)ds.
\end{equation}

The paper is organized as follows. In Section \ref{preliminaries} we introduce some notation, recall some results in Lebesgue spaces and prove their extensions to Sobolev-Lorentz spaces. These extensions (and interpolation) will play a key role in Section \ref{mainsection}, where we prove our main results: global existence, scattering, decay properties, asymptotic stability, and existence of asymptotically self-similar solutions for the integral equation $\eqref{eqintds}$.


\section{Notation and Preliminaries}\label{preliminaries}

Let us begin this section by introducing the notation used throughout the
paper. We use $C$ or $c$ to denote various constants that may vary line by line. We
denote by $\|\cdot\|_{L^p}$, $1\leq p\leq\infty$, the usual Lebesgue
$L^p$-norm. The Fourier transform of a function $f=f(x)$, is defined by
$$
(\mathcal{F}f)(\xi)=\widehat{f}(\xi)=\int_{\mathbb{R}^n}e^{-ix\cdot\xi}f(x)dx.
$$
The inverse Fourier transform of a function $g=g(\xi)$ is denoted by
$(\mathcal{F}^{-1}g)(x)=\check{g}(x)$. In $\mathcal{S}'(\re^n)$ (the space of
tempered distributions) the Fourier transform is understood in the usual
sense. $\mathcal{S}(\re^n)$ will denote the class of all Schwartz functions.

The standard Lorentz space is denoted by $L^{(p,q)}$, $0< p,q\leq \infty $%
. In particular, (the weak Lebesgue spaces) $L^{(p,\infty )}=L^{(p,\infty )}(%
\mathbb{R}^{n})$, $1\leq p<\infty $, is defined as
\begin{equation*}
L^{(p,\infty )}=\left\{ f:\mathbb{R}^{n}\rightarrow \mathbb{C}\;\mathrm{%
measurable}\;;\Vert f\Vert _{L^{(p,\infty )}}^{\ast }:=\sup\limits_{\lambda
>0}\lambda {\alpha (\lambda ,f)}^{1/p}<\infty \right\} ,
\end{equation*}%
where
\begin{equation*}
{\alpha (\lambda ,f)}=\mu \big(\{x\in {\mathbb{R}}^{n};|f(x)|>\lambda \}\big)%
,\ \mbox{and}\ \mu \;\;\text{is the Lebesgue measure}.
\end{equation*}%
The quantity $\Vert \cdot \Vert _{L^{(p,\infty )}}^{\ast }$ is a quasi-norm
in $L^{(p,\infty )}.$ As is well-known, for $1<p<\infty $ there exists an
equivalent norm $\Vert \cdot \Vert _{L^{(p,\infty )}}$ in $L^{(p,\infty )}$
(i.e., it induces the same topology than $\Vert \cdot \Vert _{L^{(p,\infty
)}}^{\ast }$), such that $L^{(p,\infty )}$ becomes a Banach space (see, for
instance, Remark 1.4.12 in \cite{Gr}). More precisely, we can define $\Vert
\cdot \Vert _{L^{(p,\infty )}}$ as%
\begin{equation*}
\left\Vert f\right\Vert _{L^{(p,\infty )}}=\sup_{t>0}t^{\frac{1}{p}}f^{\ast
\ast }(t)
\end{equation*}%
where
\begin{equation*}
\text{ \ \ \ }f^{\ast \ast }(t)=\frac{1}{t}\int_{0}^{t}f^{\ast }(s)ds\text{ }%
\ \text{and }\ f^{\ast }(t)=\inf \{\lambda >0:\text{ }{\alpha (\lambda ,f)}%
\leq t\text{ }\}.
\end{equation*}
Moreover,
$$
L^p\hookrightarrow L^{(p, \infty)}
$$
with continuous embedding.

We observe  that if $\ 1< p, q, r<\infty $, then the H\"older inequality
 \begin{align*}
&\|fg\|_{L^{(r, \infty)}}  \leq \|f\|_{L^{(p, \infty)}}\|g\|_{L^{(q,
\infty)}},\qquad \frac{1}{p}+\frac{1}{q}=\frac{1}{r},
\end{align*}
holds (see \cite{O}). Also, if $1\leq r<\infty$ and $1<q,p<\infty$, then the
Young inequality
 \begin{align}\label{youngineq}
&\|f\ast g\|_{L^{(q, \infty)}}  \leq C\|g\|_{L^{(r, \infty)}}\|f\|_{L^{(p, \infty)}},
\qquad \frac{1}{q}=\frac{1}{p}+\frac{1}{r}-1,
\end{align}
is valid (see also \cite[pages 21 and 73]{Gr}), where $L^{(r,\infty)}$ should be replaced by $L^{r}$ when $r=1$.

If $(X_1,X_2)$ is a pair of compatible Banach spaces, $\theta\in(0,1)$, and $1\leq q\leq \infty$, let $(X_1,X_2)_{\theta,\infty}$ denote the interpolation space with respect to the couple $(X_1,X_2)$ using the $K$-Method (see e.g., \cite{BeL}, \cite{Gr}, or \cite{Tr}). We recall that $L^{(p,q)}$ can be defined as an interpolation between two Lebesgue spaces. More precisely, we have:

\begin{theorem}[Interpolation theorem] \label{teoint}
Given $0<p_0<p_1 \leq \infty$, for any $ p,q$ and $\theta$ such that $p_0<q\leq \infty$, $\frac{1}{p}=\frac{1-\theta}{p_0}
+\frac{\theta}{p_1}$ and $0<\theta<1$ we have:
$$(L^{p_0},L^{p_1})_{\theta,q}=L^{(p,q)} \
\text{ with }\  \|f\|_{(L^{p_0},L^{p_1})_{\theta,q}}=\|f\|_{L^{(p,q)}},$$
where, for $q<\infty$,
$$
(L^{p_0},L^{p_1})_{\theta,q}=\{a\; is \text{ Lebesgue measurable};\|a\|_{(L^{p_0},L^{p_1})_{\theta,q}}:=\left(\int_{0}^{\infty}
t^{-\theta}k(t,a)^q \frac{dt}{t}\right)^{\frac{1}{q}}<\infty\}
$$
and
$$(L^{p_0},L^{p_1})_{\theta,\infty}=\{a\; is \text{ Lebesgue measurable};\|a\|_{(L^{p_0},L^{p_1})_{\theta,\infty}}
:=\sup\limits_{t>0}t^{-\theta}k(t,a)<\infty\},
$$
with
$$k(t,a)=\inf\limits_{a=a_0+a_1}(\|a_0\|_{L^{p_0}}+t\|a_1\|_{L^{p_1}}).$$
\end{theorem}

\begin{proof}
 We refer the reader to \cite[Theorem 5.2.1]{BeL}.
\end{proof}

\begin{remark}\label{remark2}
Theorem \ref{teoint} also holds if we replace the spaces $L^{p_k}=L^{p_k}(\R^n)$ by $L^{p_k}(\mathcal{B})=L^{p_k}(\R^n;\mathcal{B})$, the space of all measurable functions with values in the Banach space $\mathcal{B}$. In this case, $L^{p,q}(\mathcal{B})=L^{p,q}(\R^n;\mathcal{B})$ is the interpolation space.
\end{remark}

As usual, the operators $J^{-s}$ and $\Lambda^{-s}$ will denote the Bessel and Riesz potentials of order $s$, thus
$$
J^sf(x)=\{(1+|\xi|^2)^{s/2}\widehat{f}\}^{\vee}(x)
$$
and
$$
\Lambda^sf(x)=(|\xi|^{s}\widehat{f})^{\vee}(x).
$$
The base space we are interested in is presented next. We define the homogeneous Sobolev-Lorentz space $\dot{H}^s_{p,\infty}$  to be the set of all tempered distributions $f$ such that $\Lambda^sf$  belongs to $L^{(p,\infty)}$, that is,
$$
\dot{H}^s_{p,\infty}:=\{ f\in \mathcal{S}';\; \|\Lambda^sf\|_{L^{(p,\infty)}}<\infty\}.
$$
The inhomogeneous space ${H}^s_{p,\infty}$ is defined in a similar fashion by setting
$$
{H}^s_{p,\infty}:=\{ f\in \mathcal{S}';\; \|J^sf\|_{L^{(p,\infty)}}<\infty\}.
$$

Let us recall the Littlewood-Paley theory:
let $\widehat{\varphi}\in C_0^\infty(\mathbb{R}^n)$ be a function satisfying
$0\leq\widehat{\varphi}\leq1$, $\widehat{\varphi}=1$ if $|\xi|\leq1$, and
$\widehat{\varphi}=0$ if $|\xi|>2$. Define
$$
\widehat{\psi}(\xi)=\widehat{\varphi}(\xi)-\widehat{\varphi}(2\xi), \qquad
\widehat{\psi}_j(\xi)=\widehat{\psi}(2^{-j}\xi), \quad j\in \mathbb{Z},
$$
so that
$$
\sum_{j\in\mathbb{Z}}\widehat{\psi}_j(\xi)=1, \;\;\xi\neq0, \quad {\rm
and}\quad {\rm supp}\,(\widehat{\psi}_j)\subset \{2^{j-1}\leq|\xi|\leq
2^{j+1}\}.
$$
The Littlewood-Paley multipliers $\Delta_j$ are defined as
\begin{equation}\label{defdelta}
\Delta_jf=(\widehat{\psi}_j\widehat{f})^{\vee}=\psi_j\ast f, \qquad j\in
\mathbb{Z}.
\end{equation}
Also, let $\widehat{\eta}$ be another smooth function supported in
$\{1/4<|\xi|< 4\}$ such that $\widehat{\eta}= 1$ on supp$(\widehat{\psi})$.
We define $\widetilde{\Delta}_j$ like $\Delta_j$ with $\eta$ instead of
$\psi$. Thus, the identity
\begin{equation}\label{deltadelta}
 \widetilde{\Delta}_j\Delta_j=\Delta_j
\end{equation}
holds. With this notation in hand, the {\it Littlewood-Paley function} defines an equivalent norm in the Lebesgue spaces. That is to say:

\begin{theorem}[Littlewood-Paley]\label{LP}
Let  $1<p<\infty$. Then there exist positive constants $c_p$ and $C_p$ such that, for any $f \in L^p(\R^n)$,
$$c_p\|f\|_{p}\leq \|\big(\sum_j|\Delta_j f|^2\big)^{1/2}\|_{p}\leq C_p\|f\|_{p}. $$
\end{theorem}
\begin{proof}
See Theorem 5.1.2 in \cite{Gr}.
\end{proof}

Next, we recall that the Hardy-Littlewood maximal function is defined by
\begin{equation} \label{defHL}
Mf(x)=\sup\limits_{r>0}\frac{1}{\mu(B_r)}\int_{B_r}|f(x-y)|dy.
\end{equation}
Here $f$ is any locally integrable function and $B_r$ is the Euclidean ball of radius $r$ centered at the origin. A well known property of $M$ is given below.

\begin{theorem}\label{HL}
The operator $M$ is weak-type $(1,1)$ and strong-type $(p,p)$, $1<p \leq \infty$.
\end{theorem}
\begin{proof}
See Theorem 2.1.6 in \cite{Gr}.
\end{proof}

In what follows we denote by $L^p(l^2)$, $1<p<\infty$, the space of all sequences $(f_k)$ of measurable functions on $\R^n$ satisfying
$$
\|(f_k)\|_{L^p(l^2)}=\left\|\left( \sum_k|f_k|^2 \right)^{\frac{1}{2}} \right\|_{L^p}<\infty.
$$

Given  a linear operator $T$ acting on $L^p(\R^n)$ and taking values
in the set of measurable functions, we can define an $l^2$-valued extension, said as $\overrightarrow{T}$, by setting
$$
\overrightarrow{T}(\{f_k\}_k):=\{Tf_k\}_k.
$$

\begin{lemma}\label{extentionlemma}
If $T$ is a bounded liner operator from $L^p$ to $L^q$, $1<p,q<\infty$, then $\overrightarrow{T}$ is also a bounded linear operator from $L^p(l^2)$ to $L^q(l^2)$.
\end{lemma}
\begin{proof}
See Theorem 4.5.1 in \cite{Gr}.
\end{proof}

As an immediate consequence of  Lemma \ref{extentionlemma} and Theorem \ref{HL} the Hardy-Littlewood maximal function has a bounded $l^2$-valued extension (See Example 9.5.9 in \cite{Gr1}). More precisely.

\begin{corollary}
If $\overrightarrow{M}$ denotes the $l^2$-valued extension of $M$ then, for $1<p<\infty$,
\begin{equation} \label{vectorHL}
\|\overrightarrow{M}(\{f_k\})\|_{L^p(l^2)}\leq C_p\|\{f_k\}\|_{L^p(l^2)},
\end{equation}
for some constant $C_p>0$.
\end{corollary}

By using the real interpolation method (see, for instance, \cite{BeL} or \cite{Gr} for details) we will extend the previous results to weak Lebesgue spaces. We start with the Littlewood-Paley inequality.

\begin{lemma}\label{extension}
Let  $1<p<\infty$. Then there exist positive constants $c_p$ and $C_p$ such that, for any $f \in L^{(p,\infty)}(\R^n)$,
\begin{equation}\label{LPweal}
c_p\|f\|_{L^{(p,\infty)}}\leq \|\big(\sum_j|\Delta_j f|^2\big)^{1/2}\|_{L^{(p,\infty)}}\leq C_p\|f\|_{L^{(p,\infty)}}.
\end{equation}
 \end{lemma}
\begin{proof}
We will present two different proofs based on the real interpolation method.

{\it First Proof.}
Let us begin by proving the second inequality. Define the following linear operator
\begin{equation}
\begin{split}
&T:L^{p}(\R^n)\rightarrow L^{p}(l^2)\\
&
T(f)=(\Delta_j f)_j.
\end{split}
\end{equation}
From Theorem \ref{LP} we know that $T$ is well-defined. Now,
taking $1<p_0<p_1<\infty$ we also have from Theorem \ref{LP} that $T$ is bounded from $L^{p_k}(\R^n)$ to $L^{p_k}(l^2), k=0,1.$ Applying the real interpolation method we obtain that $T$ is also bounded from $(L^{p_0}(\R^n),L^{p_1}(\R^n)_{\theta,\infty}$
to $(L^{p_0}(l^2),L^{p_1}(l^2)_{\theta,\infty}$ where $0<\theta<1$. Since $$(L^{p_0}(\R^n),L^{p_1}(\R^n)_{\theta,\infty}=L^{(p,\infty)}(\R^n)$$ and $$(L^{p_0}(l^2),L^{p_1}(l^2)_{\theta,\infty}=L^{(p,\infty)}(l^2),$$ where $\frac{1}{p}=\frac{\theta}{p_0}+\frac{1-\theta}{p_1}$, we get the result.

To prove the first inequality we define, for each $1<p<\infty$, the following subspace of $L^{p}(l^2)$:
$$A_p=\{\tilde{f}\in L^{p}(l^2); \exists f \in L^{p}(\R^n) \ s.t.\ \tilde{f}= (\Delta_j f)_j\}.$$
From Theorem \ref{LP} we know that $(A_p,\|\cdot\|_{L^{p}(l^2)})$ is a Banach space. Now consider the linear operator $T_1$ defined by
\begin{equation*}
\begin{split}
&T_1:A_{p}\rightarrow L^{p}(\R^n)\\
&
T_1(\tilde{f}):=\sum_j\Delta_j f=f.
\end{split}
\end{equation*}
Using the same idea as before we take $1<p_0<p_1<\infty$. From Theorem \ref{LP} we obtain that
$T_1$ is bounded from $A_{p_k}$ to $L^{p_k}(\R^n), k=0,1.$ From real interpolation we have that $T_1$ is bounded from $(A_{p_0},A_{p_1})_{\theta,\infty}$
to $(L^{p_0}(\R^n),L^{p_1}(\R^n)_{\theta,\infty}$. Thus, it suffices to prove that $(A_{p_0},A_{p_1})_{\theta,\infty}=A_{p,\infty}$,
where
 $$A_{p,\infty}=\{\tilde{f}\in L^{(p,\infty)}(l^2); \exists f \in L^{(p,\infty)}(\R^n) \ s.t.\ \tilde{f}= (\Delta_j f)_j\}.$$

Let $\tilde{f} \in A_{p,\infty}$. Then $\tilde{f}=(\Delta_j f)_j$ for some $f \in L^{(p,\infty)}(\R^n)$. Therefore,
\begin{equation*}
\begin{split}
&\|{\tilde{f}}\|_{(A_{p_0},A_{p_1})_{\theta,\infty}}=\sup\limits_{t>0}t^{-\theta}\inf\limits_{\tilde{f}=\tilde{f_0}+\tilde{f_1}}(\|{\tilde{f_0}}\|_{L^{p_0}(l^2)}+\|{\tilde{f_1}}\|_{L^{p_1}(l^2)}).
\end{split}
\end{equation*}
Since $\tilde{f_0} \in A_{p_0}$ and $\tilde{f_1} \in A_{p_1}$,
$$\tilde{f_0}=(\Delta_jf_0)_j, \text{ for some }f_0 \in L^{p_0}(\R^n)$$
and
$$\tilde{f_1}=(\Delta_jf_1)_j, \text{ for some }f_1 \in L^{p_1}(\R^n).$$
So,
\begin{equation*}
\begin{split}
&\|{\tilde{f}}\|_{(A_{p_0},A_{p_1})_{\theta,\infty}}=\sup\limits_{t>0}t^{-\theta}\inf_{(\Delta_j f)_j=(\Delta_jf_0)_j+(\Delta_jf_1)_j}(\|{(\Delta_jf_0)_j}\|_{L^{p_0}(l^2)}+\|{(\Delta_jf_1)_j}\|_{L^{p_1}(l^2)})
\end{split}
\end{equation*}
and, in view of  Theorem \ref{LP},
\begin{equation*}
\begin{split}
\|{\tilde{f}}\|_{(A_{p_0},A_{p_1})_{\theta,\infty}}&\leq C \sup\limits_{t>0}t^{-\theta}\inf\limits_{f=f_0+f_1}(\|{f_0}\|_{L^{p_0}(\R^n)}+\|{f_1}\|_{L^{p_1}(\R^n)})\\
&=C\|f\|_{L^{(p,{\infty})}(\R^n)}<\infty.
\end{split}
\end{equation*}
On the other hand, let $\tilde{f} \in (A_{p_0},A_{p_1})_{\theta,\infty}$. Then ${\tilde{f}=\tilde{f_0}+\tilde{f_1}}$, with $\tilde{f_0} \in A_{p_0}$ and $\tilde{f_1} \in A_{p_1}.$ We will show that $\tilde{f} \in A_{p,\infty}.$ In fact,
\begin{equation*}
\|{\tilde{f}}\|_{A_{p,\infty}}= \sup\limits_{t>0}t^{-\theta}\inf\limits_{\tilde{f}=F_0+F_1}(\|{F_0}\|_{L^{p_0}(l^2)}+\|{F_1}\|_{L^{p_1}(l^2)}).
\end{equation*}
Since $ A_{p_0} \subset L^{p_0}(l^2)$ and $ A_{p_1} \subset L^{p_1}(l^2)$ we conclude that
\begin{equation*}
\|{\tilde{f}}\|_{A_{p,\infty}}\leq C \sup\limits_{t>0}t^{-\theta}\inf\limits_{{\tilde{f}=\tilde{f_0}+\tilde{f_1}}}(\|\tilde{f_0}\|_{L^{p_0}(l^2)}+\|\tilde{f_1}\|_{L^{p_1}(l^2)}),
\end{equation*}
where $\tilde{f_0} \in A_{p_0}$ and $\tilde{f_1} \in A_{p_1}.$ Therefore
\begin{equation*}
\|{\tilde{f}}\|_{A_{p,\infty}}\leq C\|{\tilde{f}}\|_{(A_{p_0},A_{p_1})_{\theta,\infty}}<\infty,
\end{equation*}
which shows the desired result.\\

{\it Second Proof.} Define the norms
$$
\trinorm f\trinorm_{L^p}= \|\big(\sum_j|\Delta_j f|^2\big)^{1/2}\|_{L^{p}} \quad \mbox{and} \quad  \trinorm f\trinorm_{L^{(p,\infty)}}=\|\big(\sum_j|\Delta_j f|^2\big)^{1/2}\|_{L^{(p,\infty)}}.
$$
From Theorem \ref{LP}, $\trinorm \cdot\trinorm_{L^p}$ is an equivalent norm in $L^p$, $1<p<\infty$. This means that the identity operator $I:(L^{p_k}, \trinorm \cdot\trinorm_{L^{p_k}})\to (L^{p_k}, \| \cdot\|_{L^{p_k}})$ is continuous for any $1<p_1<p_2<\infty$. Hence, by real interpolation, $I$ also is continuous from $(L^{(p,\infty)}, \trinorm \cdot \trinorm_{L^{(p,\infty)}})$ to  $(L^{(p,\infty)}, \| \cdot \|_{L^{(p,\infty)}})$, $p_1<p<p_2$, which means that the first inequality in \eqref{LPweal} holds.
 Here, to see that the interpolation space between $(L^{p_1}, \trinorm \cdot\trinorm_{L^{p_1}})$ and $(L^{p_2}, \trinorm \cdot\trinorm_{L^{p_2}})$ is indeed $(L^{(p,\infty)}, \trinorm \cdot \trinorm_{L^{(p,\infty)}})$ it is sufficient to recall that being $L^p$ a retract of $L^p(l^2)$, $(L^{p_1},L^{p_2})_{\theta,\infty}$ is also a retract of  $(L^{p_1}(l^2),L^{p_2}(l^2))_{\theta,\infty}$ (see Theorems 6.4.2 and 6.4.3 in \cite{BeL}).

 The second inequality is obtained in a similar fashion, which concludes the proof.
\end{proof}

\begin{remark}\label{equinorm}
In view of Lemma \ref{extension} one can see that the expression
$$
\|\{2^{js}\Delta_j f\}\|_{L^{(p,\infty)}(l^2)}=  \left\|\left( \sum_j2^{2js}|\Delta_j f|^2\right)^{1/2}\right\|_{L^{(p,\infty)}}
$$
defines an equivalent norm in $\dot{H}^s_{p,\infty}$. Indeed, note that $2^{js}\Delta_j f=\Delta_j^\sigma f_s$, where $f_s=\Lambda^sf$ and $\Delta_j^\sigma$ is the ``Littlewood-Paley multiplier'' given by $\widehat{\Delta_j^\sigma f}(\xi)=\widehat{\sigma}(2^{-j}\xi)\widehat{f}$, with $\widehat{\sigma}(\xi)=|\xi|^{-s}\widehat{\psi}_j(\xi)$. Thus applying Lemma \ref{extension} with $\Delta_j^\sigma$ instead of $\Delta_j$, we obtain
$$
\|\{2^{js}\Delta_j f\}\|_{L^{(p,\infty)}(l^2)}= \|\{\Delta_j^\sigma f_s \}\|_{L^{(p,\infty)}(l^2)}\sim \|f_s\|_{L^{(p,\infty)}}.
$$
The interested reader will find the details (for $L^p$ instead of $L^{(p,\infty)}$) in \cite[Theorem 6.2.7]{Gr1}. In particular, this shows that the space $\dot{H}^s_{p,\infty}$ can be obtained as an interpolation space between two Sobolev spaces. More precisely
$$
\dot{H}^s_{p,\infty}=(\dot{H}^s_{p_0},\dot{H}^s_{p_1})_{\theta,\infty}, \qquad p_0\neq p_1, \quad 0<\theta<1,
$$
where $\frac{1}{p}=\frac{1-\theta}{p_0}+\frac{\theta}{p_1}$ (see Theorem 1 in \cite[page 184]{Tr}).
\end{remark}

In what follows, the space $L^{(p,\infty)}(l^2)$ is defined similarly to $L^{p}(l^2)$ when replacing $\|\cdot\|_{L^p}$ by $\|\cdot\|_{L^{(p,\infty)}}$.

\begin{lemma}\label{extensionM}
	If $\overrightarrow{M}$ denotes the $l^2$-valued extension of $M$ then, for $1<p<\infty$,
and $\{f_k\} \in L^{(p,\infty)}(l^2)$,
 $$
 \|\overrightarrow{M}(\{f_k\})\|_{L^{(p,\infty)}(l^2)}\leq C_p\|\{f_k\}\|_{L^{(p,\infty)}(l^2)}.
 $$
\end{lemma}
\begin{proof}
It suffices to  follow the  ideas  in the proof of the previous lemma. From inequality \eqref{vectorHL} we have that
$\overrightarrow{M}$ is bounded from $L^{p_0}(l^2)$ to $L^{p_0}(l^2)$ and from $L^{p_1}(l^2)$ to $L^{p_1}(l^2)$, where $1<p_0<p_1<\infty$. From real interpolation we get the result.
\end{proof}

Next result is an adapted version to weak Lebesgue spaces of the Sobolev embedding. Since $f(x)=\{|\xi|^{-s}\widehat{f}_s\}^\vee(x)= I_s(f_s)(x),$
where $f_s(x)=\Lambda^sf(x)$ and $I_s=\Lambda^{-s}= |x|^{-(n-s)}*$, it can be proved by using Young's inequality for convolution operators in Lorentz spaces (see \eqref{youngineq} and \cite[page 73]{Gr}).

\begin{lemma}\label{lemaun}
Let $s$ be a real number with $0<s<n$ and let $1\leq p<q<\infty$ satisfy \begin{equation}\label{Iequi3}
s=\frac{n}{p}-\frac{n}{q}.
\end{equation}
Then there exists a positive constant $C=C(n,s,p)$ such that, for any $f\in \dot{H}^s_{p,\infty}$,
\begin{equation}\label{aun}
 \|f\|_{L^{(q, \infty)}}\leq  C\|\Lambda^s f\|_{L^{(p, \infty)}}.
\end{equation}
 \end{lemma}

\begin{remark}\label{remark1}
As a consequence of the previous lemma we obtain that
$$\|u\|_{L^{(\frac{\gamma(\alpha+2)}{\alpha}, \infty)}}\leq  C\|\Lambda^s u\|_{L^{(\alpha+2, \infty)}},$$
for  $0<\alpha<\gamma$ and $s=\frac{n(\gamma -\alpha)}{\gamma(\alpha+2)}.$
\end{remark}

In order to estimate the nonlinear part of our problem we also need the following lemma.

\begin{lemma}\label{lemadeux}
Let $0<\alpha<\gamma$ and suppose that $s$ defined by $s=\frac{n(\gamma -\alpha)}{\gamma(\alpha+2)}$ satisfies $0< s< 1$. Then
\begin{equation}\label{deux}
 \|\Lambda^s(|u|^\alpha u)\|_{L^{(\frac{\alpha+2}{\alpha+1}, \infty)}}\leq C\|u\|_{L^{(\alpha+2, \infty)}}^\alpha \|\Lambda^s u\|_{L^{(\alpha+2, \infty)}}
\end{equation}
and
\begin{equation}\label{deux1}
 \|\Lambda^s(|u|^\gamma u)\|_{L^{(\frac{\alpha+2}{\alpha+1}, \infty)}} \leq C\|\Lambda^s u\|^{\gamma+1}_{L^{(\alpha+2, \infty)}}.
\end{equation}
 \end{lemma}

 To prove Lemma \ref{lemadeux} we  need two additional results. The first one is the Leibniz rule for weak Lebesgue spaces.

 \begin{lemma}\label{lematrois}
Suppose that $F\in C^1(\C,\C)$, $s\in (0,1)$, and assume that $1<p,q,r<\infty$ satisfy $\frac{1}{r}=\frac{1}{p}+\frac{1}{q}$. Then,
\begin{equation}\label{trois}
 \|\Lambda^s F(u)\|_{L^{(r, \infty)}}\leq C\|F'(u)\|_{L^{(p, \infty)}} \|\Lambda^s u\|_{L^{(q, \infty)}},
\end{equation}
as long as the right-hand side is finite.
 \end{lemma}

\begin{proof}
Since the ideas are the same as those in \cite{ChW} we will only outline the proof making the necessary adaptations to Lorentz spaces.

At first we observe that
$$\Delta_jF(u)(x)=\int_{\R^n}\left[\int_0^1 F'\Big(tu(y)+(1-t)u(x)\Big)dt\right]\big(u(y)-u(x)\big)\psi_j(x-y)dy,$$
where $\Delta_j$ and $\psi_j$ were defined in \eqref{defdelta}.
Using that
$$
\left|\int_0^1 F'\Big(tu(y)+(1-t)u(x)\Big)dt\right|\leq 2M(F'(u))(x)
$$
and decomposing $u=\sum_k \widetilde{\Delta}_k\Delta_k u$, we obtain
\begin{equation}\label{trois1}
|\Delta_jF(u)(x)|\leq C M(F'(u))(x)\cdot \sum_{k=-\infty}^\infty \int_{\R^n}| \widetilde{\Delta}_k\Delta_k u(y)- \widetilde{\Delta}_k\Delta_k u(x)||\psi_j(x-y)|dy,
\end{equation}
where $M$ is the Hardy-Littlewood maximal function. Now we break the sum over $k$ into the cases $k<j$ and $k\geq j$. By using the properties of $\psi_j$ and $\Delta_j$, we see that
\begin{equation}\label{trois2}
 \sum_{k<j} \int_{\R^n}| \widetilde{\Delta}_k\Delta_k u(y)- \widetilde{\Delta}_k\Delta_k u(x)||\psi_j(x-y)|dy\leq C\sum_{k<j} 2^{k-j}{M}^2 \Delta_k u(x),
 \end{equation}
where ${M}^2(\Delta_k u(x))=M\circ M(\Delta_k u(x)).$ In a similar fashion,
\begin{equation}\label{trois3}
\sum_{k\geq j} \int_{\R^n}| \widetilde{\Delta}_k\Delta_k u(y)- \widetilde{\Delta}_k\Delta_k u(x)||\psi_j(x-y)|dy\leq C \sum_{k\geq j}{M}^2 \Delta_k u(x).
\end{equation}
Inserting \eqref{trois2} and \eqref{trois3} into \eqref{trois1}, substituting $j=k-m$ and applying Minkowski's inequality, we have
\begin{equation*}
\big(\sum_{j=-\infty}^\infty 2^{2js}|\Delta_jF(u)(x)|^2\big)^{1/2}\leq C MF'(u)(x)\sum_{m=-\infty}^\infty
2^{-\epsilon |m|}\big(\sum_{k=-\infty}^\infty 2^{2ks}|{M}^2\Delta_ku(x)|^2\big)^{1/2},
\end{equation*}
where $\epsilon =2 \min(s,1-s)>0.$

In view of Remark \ref{equinorm},  Holder's inequality and Lemma \ref{extensionM}, we then deduce
\[
\begin{split}
\|\Lambda^sF(u)\|_{L^{(r, \infty)}}& \leq C
\|MF'(u)\|_{L^{(p, \infty)}}
\|\big(\sum_{k=-\infty}^\infty 2^{2ks}|{M}^2\Delta_ku|^2\big)^{1/2}\|_{L^{(q, \infty)}}\\
&= C\|MF'(u)\|_{L^{(p, \infty)}}\|\overrightarrow{M}^2\{2^{ks}\Delta_j u\}\|_{L^{(q,\infty)}(l^2)}\\
&\leq  C\|MF'(u)\|_{L^{(p, \infty)}}\|\{2^{ks}\Delta_j u\}\|_{L^{(q,\infty)}(l^2)}.
\end{split}
\]
The conclusion now follows as an application of Remark \ref{equinorm}.
\end{proof}

Next Lemma is similar to Lemma A.2 in \cite{K1}.

 \begin{lemma}\label{lemaquatre}
Suppose that $F\in C^1(\C,\C)$ satisfies $F(0)=0$ and $|F'(x)|\leq c|x|^{k-1},\ \  k\geq 1.$
 If $s\in [0,1]$, then
 \begin{equation}\label{key}
  \|\Lambda^s F(u)\|_{L^{(r, \infty)}}\leq c\|u\|^{k-1}_{L^{(p, \infty)}} \|\Lambda^s u\|_{L^{(q, \infty)}}
 \end{equation}
 where
 $1<p,q,r<\infty, \ \frac{1}{r}=\frac{k-1}{p}+\frac{1}{q}$  and the constant $c$ depends on $s,p,q,r.$
 \end{lemma}
 \begin{proof}
 Let us start with the case $s=0.$ Since $F\in C^1(\C,\C)$ and $F(0)=0$ we have that $$\left|\frac{F(u)-F(0)}{u-0}\right|\leq c|u|^{k-1}.$$
 Now from Holder's inequality,
 $$
 \|F(u)\|_{L^{(r, \infty)}}\leq c\|u\|^{k-1}_{L^{(\tilde p(k-1), \infty)}} \|u\|_{L^{(q, \infty)}} \text{ where }\ \frac{1}{r}=\frac{1}{\tilde p}+\frac{1}{q}.
 $$
 Taking $\tilde p (k-1)=p$ we get the result.
 To solve the case $s=1$ we use the same idea together with the facts that $$\partial F(u)=F'(u) \partial u\  \text{ and } \  \|\partial u \|_{L^{(p, \infty)}}=\|\Lambda u \|_{L^{(p, \infty)}}.$$
Assume now $0<s<1.$ From  Lemma \ref{lematrois}, we get
 $$\|\Lambda^s F(u)\|_{L^{(r, \infty)}}\leq c\|F'(u)\|_{L^{(\tilde{p}, \infty)}} \|\Lambda^s u\|_{L^{(q, \infty)}} \text{ for } \frac{1}{r}=\frac{1}{\tilde{p}}+\frac{1}{q}.$$
 Using the hypothesis on $F$ and Holder's inequality, we obtain
 $$\|\Lambda^s F(u)\|_{L^{(r, \infty)}}\leq c\|u\|^{k-1}_{L^{(\tilde{p}({k-1}), \infty)}} \|\Lambda^s u\|_{L^{(q, \infty)}}.$$
By taking $\tilde p (k-1)=p$ we get the desired inequality.
 \end{proof}

With Lemma \ref{lemaquatre} in hand we are able to prove Lemma \ref{lemadeux}.

\begin{proof}[Proof of Lemma \ref{lemadeux}]
To prove \eqref{deux} we only need to choose $F(x)=|x|^\alpha x, \ r=\frac{\alpha+2}{\alpha+1},\ p=q=\alpha+2$,$\ k-1=\alpha$ and apply Lemma \ref{lemaquatre}.
To prove \eqref{deux1} we first note that, as above, an application of Lemma \ref{lematrois} gives,
$$
 \|\Lambda^s(|u|^\gamma u)\|_{L^{(\frac{\alpha+2}{\alpha+1}, \infty)}}\leq C \|u\|^\gamma_{L^{(\frac{\gamma(\alpha+2)}{\alpha},\infty)}}  \|\Lambda^s u\|_{L^{(\alpha+2, \infty)}}.
$$
An application of Remark \ref{remark1} establishes the desired inequality.
\end{proof}

\begin{remark}
In \cite{K1} it was proved (in the $L^p$ level) that \eqref{key} also holds if $F\in C^m(\mathbb{C}, \mathbb{C})$ satisfies $|D^iF(x)|\leq |x|^{k-i},$ $i=1,\ldots,m$ for some $k\geq m$, and $s\in[0,m]$. The proof relies on the Gagliardo-Nirenberg type inequality
\begin{equation}\label{gnlp}
\|\Lambda^sf\|_{L^p}\leq C\|\Lambda^{s_0}f\|_{L^{p_0}}^{1-\theta}C\|\Lambda^{s_1}f\|_{L^{p_1}}^\theta,
\end{equation}
where $\theta\in(0,1)$, $s=(1-\theta)s_0+\theta s_1$, and $\frac{1}{p}=\frac{1-\theta}{p_0}+\frac{\theta}{p_1}$. Here, since we do not know if \eqref{gnlp} holds in the $L^{(p,\infty)}$ level, we are unable to prove a similar result. The drawback is that we will not reach all ranges of $\alpha$ and $\gamma$ as in \cite{W}.
\end{remark}

Finally, next lemma establishes the boundedness of the linear group $U(t)$ in the weak
Lebesgue spaces.

\begin{lemma}\label{lema2.1}
	Let  $1<p<2$. If $p'$ is such that $\frac{1}{p}+\frac{1}{p'}=1$, then there
	exists a constant $C=C(n,p)>0$ such that
	\begin{equation}\label{a1}
	\|U(t)\phi\|_{L^{(p', \infty)}}\leq C t^{-\frac{n}{d}(\frac{2}{p}-1)}\|\phi\|_{L^{(p, \infty)}},
	\end{equation}
	for all $\phi \in L^{(p,\infty)}(\re^n)$ and $t>0$.
\end{lemma}
\begin{proof}
	We refer the reader to \cite{BP} (see also \cite{SFR}) for a proof of this lemma.
\end{proof}

\section{Main Results}\label{mainsection}

In this section we will  state and prove our main results. We follow the ideas in \cite{W} where the author proves existence of global solutions for small initial data with respect to a norm which is related to the structure of the  two-power nonlinear Shcr\"odinger equation. Our results extend the ones in \cite{W} since weak-$L^p$ spaces contain Lebesgue's spaces. At first let us define the function spaces where the solutions will be obtained.

\begin{definition}\label{def1}
 Given positive numbers $s,\beta, \delta$ and $M$, let $X_M=X_M(s,\beta, \delta)$ be the set of Bochner-measurable functions $u:(0,+\infty) \rightarrow H^s_{\alpha+2,\infty}$  such that

\begin{equation} \label{d3}
\|u\|_{\beta}:=\sup_{t>0}t^{\beta}
\|u(t)\|_{L^{(\alpha+2,\infty)}}\leq M,
\end{equation}
and
\begin{equation} \label{d4}
\|u\|_{\delta,s}:=\sup_{t>0}t^{\delta}
\|\Lambda^su(t)\|_{L^{(\alpha+2,\infty)}}\leq M.
\end{equation}

\end{definition}

It is not difficult to see  that $(X_M,d)$ is a nonempty complete metric space endowed with the distance
$$
d(u,v):=\sup_{t>0}t^{\beta} \|u(t)-v(t)\|_{L^{(\alpha+2,\infty)}}.
$$

For $\rho>0,$ we also define the initial data class $\mathcal{I}_{\rho
}\mathcal{=I}_{\rho}(s,\beta,\delta)$ as follows
\begin{equation}
\mathcal{I}_{\rho}=\{u_{0}\in\mathcal{S}^{\prime};U(t)u_{0}\in X_{\rho
}(s,\beta,\delta)\}.\label{initial data-1}%
\end{equation}

Our first result reads as follows.

\begin{theorem}[Global Existence and Uniqueness]\label{globaltheorem1}
Assume $0<max\{1, \alpha\}<\gamma$. Define
\begin{equation}\label{s}
s=\frac{n(\gamma-\alpha)}{\gamma(\alpha +2)}
\end{equation}
and suppose that $0<s<1$ and
\begin{equation}\label{alpha}
\frac{\alpha+2}{\alpha+1}< \frac{n\alpha}{d}< \alpha+2.
\end{equation}

Consider the positive numbers $\beta$ and $\delta$ defined by
\begin{equation} \label{betadelta}
\beta=\frac{1}{\alpha}-\frac{n}{d(\alpha +2)} \ \text{  and  }\ \
\delta=\frac{1}{\gamma}+\frac{s}{d}-\frac{n}{d(\alpha +2)}.
\end{equation}
Let $\rho>0$ and $M>0$ be such that
\begin{equation}\label{a}
\rho +|a|C_1M^{\alpha+1}+|b|C_2M^{\gamma+1}\leq M
\end{equation}
and
\begin{equation}\label{a_1}
\rho +|a|C_3M^{\alpha+1}+|b|C_4M^{\gamma+1}\leq M
\end{equation}
for some positive constants $C_1,C_2,C_3,C_4$ given in the calculations below, and assume that
\begin{align}\label{b}
|a|C_1M^{\alpha}+|b|C_2M^{\gamma}<1.
\end{align}

If $u_0\in\mathcal{I}_{\rho}$, then there exists a unique global solution of \eqref{eqintds}, say, $u \in X_M(s,\beta, \delta).$
\end{theorem}
\begin{proof}
The proof is based on the Banach fixed point theorem. We consider the
integral operator
\begin{equation} \label{Z1}
(\Phi u)(t)=U(t)u_0 + (Bu)(t),
\end{equation}
where $B$ is defined as in \eqref{defB1}.

Let $X_M=X_M(s,\beta, \delta)$ be the function space from Definition \ref{def1}.
 We will show  that $\Phi$ maps $X_M$ into itself and $\Phi:X_M\rightarrow
X_M$ is a contraction. To do that, we assume that $u, v \in X_M$ and estimate the integrals below:
\begin{align*}
J_{1}:= t^\beta |a|\int_0^t\|U(t-\tau)(|u|^\alpha u-|v|^\alpha v)(\tau)\|_{L^{(\alpha+2,\infty)}}d\tau,
\end{align*}
\begin{align*}
J_{2}:= t^\beta |b|\int_0^t\|U(t-\tau)[E(|u|^\gamma) u-E(|v|^\gamma) v](\tau)\|_{L^{(\alpha+2,\infty)}}d\tau,
\end{align*}
\begin{align*}
J_{3}:= t^\delta |a|\int_0^t\|\Lambda^s U(t-\tau)(|u|^\alpha u)(\tau)\|_{L^{(\alpha+2,\infty)}}d\tau,
\end{align*}
and
\begin{align*}
J_{4}:=  t^\delta |b|\int_0^t\|\Lambda^s U(t-\tau)[E(|u|^\gamma) u](\tau)\|_{L^{(\alpha+2,\infty)}}d\tau.
\end{align*}

For $m>0,$ recall the pointwise inequality
\begin{equation}
||u|^{m}u-|v|^{m}v|\leq C(|u|^{m}+|v|^{m})|u-v|.\text{
}\label{wellknow}%
\end{equation}
To estimate $J_{1},$ we use (\ref{wellknow}) with $m=\alpha$, Lemma
\ref{lema2.1} and Holder's inequality to get
\begin{align*}
J_{1}\leq &  |a|Ct^{\beta}\int_{0}^{t}(t-\tau)^{-\frac{n}{d}(\frac{2(\alpha
+1)}{\alpha+2}-1)}\Vert(|u|^{\alpha}u-|v|^{\alpha}v)(\tau)\Vert_{L^{(\frac
{\alpha+2}{\alpha+1},\infty)}}d\tau\\
\leq &  |a|Ct^{\beta}\int_{0}^{t}(t-\tau)^{-\frac{n}{d}\frac{\alpha}{\alpha
+2}}(\Vert u(\tau)\Vert_{L^{(\alpha+2,\infty)}}^{\alpha}+\Vert v(\tau
)\Vert_{L^{(\alpha+2,\infty)}}^{\alpha})\Vert u(\tau)-v(\tau)\Vert
_{L^{(\alpha+2,\infty)}}d\tau\\
\leq &  C|a|M^{\alpha}d(u,v)t^{1-\frac{n}{d}\frac{\alpha}{\alpha+2}%
-\alpha\beta}\int_{0}^{1}(1-s)^{-\frac{n}{d}\frac{\alpha}{\alpha+2}%
}s^{-(\alpha+1)\beta}ds\\
\leq &  |a|C_{1}M^{\alpha}d(u,v),
\end{align*}
where in the last inequality we used that $\beta(\alpha+1)<1$, $\frac{n\alpha
}{d(\alpha+2)}<1$, and
\begin{equation}
\beta\alpha+\frac{n\alpha}{d(\alpha+2)}=1.\label{alberel}%
\end{equation}

Now let us estimate $J_{2}$. Firstly, we write
\begin{equation}
E(|u|^{\gamma})u-E(|v|^{\gamma})v=E(|u|^{\gamma})(u-v)+E(|u|^{\gamma
}-|v|^{\gamma})v.\label{ineq-dif-1}%
\end{equation}
Let $z=\frac{\alpha+2}{\alpha}\gamma$ and $l=\frac{(\alpha+2)\gamma}%
{(\alpha+1)(\gamma-1)+1}.$ Observe that $z,l>1,$ $\frac{1}{l}=\frac{1}%
{\alpha+2}+\frac{\gamma-1}{z}$ and  $\frac{1}{l}+\frac{1}{z}=\frac{\alpha
+1}{\alpha+2}$.  Then,  using Holder's inequality, assumption \textbf{(H3)},
(\ref{ineq-dif-1}), (\ref{wellknow}) with $m=\gamma-1$ and afterwards
Remark \ref{remark1}, we obtain
\begin{align}
&\left\Vert E(|u|^{\gamma})u-E(|v|^{\gamma})v\right\Vert _{L^{(\frac{\alpha
+2}{\alpha+1},\infty)}}\nonumber\\
&  \leq\left\Vert E(|u|^{\gamma})(u-v)\right\Vert
_{L^{(\frac{\alpha+2}{\alpha+1},\infty)}}+\left\Vert E(|u|^{\gamma}-|v|^{\gamma
})v\right\Vert _{L^{(\frac{\alpha+2}{\alpha+1},\infty)}}\nonumber\\
&  \leq C\left\Vert E(|u|^{\gamma})\right\Vert _{L^{(\frac{\alpha+2}{\alpha
},\infty)}}\left\Vert u-v\right\Vert _{L^{(\alpha+2,\infty)}}+C\left\Vert
E(|u|^{\gamma}-|v|^{\gamma})\right\Vert _{L^{(l,\infty)}}\left\Vert v\right\Vert
_{L^{(z,\infty)}}\nonumber\\
&  \leq C\left\Vert |u|^{\gamma}\right\Vert _{L^{(\frac{\alpha+2}{\alpha
},\infty)}}\left\Vert u-v\right\Vert _{L^{(\alpha+2,\infty)}}+C\left\Vert
(|u|^{\gamma}-|v|^{\gamma})\right\Vert _{L^{(l,\infty)}}\left\Vert \Lambda
^{s}v\right\Vert _{L^{(\alpha+2,\infty)}}\nonumber\\
&  \leq C\left\Vert |u|^{\gamma}\right\Vert _{L^{(\frac{\alpha+2}{\alpha}%
,\infty)}}\left\Vert u-v\right\Vert _{L^{(\alpha+2,\infty)}}+C\left\Vert
|u|^{\gamma-1}+|v|^{\gamma-1}\right\Vert _{L^{(\frac{z}{\gamma-1},\infty
)}}\left\Vert u-v\right\Vert _{L^{(\alpha+2,\infty)}}\left\Vert \Lambda
^{s}v\right\Vert _{L^{(\alpha+2,\infty)}}\nonumber\\
&  \leq C\left\Vert u \right\Vert _{L^{(z,\infty)}}^{\gamma}\left\Vert
u-v\right\Vert _{L^{(\alpha+2,\infty)}}+C(\left\Vert u\right\Vert _{L^{(z,\infty
)}}^{\gamma-1}+\left\Vert v\right\Vert _{L^{(z,\infty)}}^{\gamma-1})\left\Vert
u-v\right\Vert _{L^{(\alpha+2,\infty)}}\left\Vert \Lambda^{s}v\right\Vert
_{L^{(\alpha+2,\infty)}}\nonumber\\
&  \leq C\left\Vert \Lambda^{s}u \right\Vert _{L^{(\alpha+2,\infty)}}^{\gamma
}\left\Vert u-v\right\Vert _{L^{(\alpha+2,\infty)}}\nonumber\\
&+C(\left\Vert \Lambda
^{s}u\right\Vert _{L^{(\alpha+2,\infty)}}^{\gamma-1}+\left\Vert \Lambda
^{s}v\right\Vert _{L^{(\alpha+2,\infty)}}^{\gamma-1})\left\Vert u-v\right\Vert_{L^{(\alpha+2,\infty)}}\left\Vert \Lambda^{s}v\right\Vert _{L^{(\alpha+2,\infty
)}}\nonumber\\
&\leq C(\Vert\Lambda^{s}u\Vert_{L^{(\alpha+2,\infty)}}^{\gamma
}+\Vert\Lambda^{s}v\Vert_{L^{(\alpha+2,\infty)}}^{\gamma})\Vert
u-v\Vert_{L^{(\alpha+2,\infty)}}.
\end{align}

From the last inequality and Lemma \ref{lema2.1}, we obtain that

\begin{align*}
J_{2} &  \leq|b|Ct^{\beta}\int_{0}^{t}(t-\tau)^{-\frac{n}{d}\frac{\alpha}%
{\alpha+2}}(\left\Vert E(|u|^{\gamma})u-E(|v|^{\gamma})v\right\Vert _{L^{(\frac{\alpha
+2}{\alpha+1},\infty)}})d\tau\\
 &\leq|b|Ct^{\beta}\int_{0}^{t}(t-\tau)^{-\frac{n}{d}\frac{\alpha}%
{\alpha+2}}(\Vert\Lambda^{s}u(\tau)\Vert_{L^{(\alpha+2,\infty)}}^{\gamma
}+\Vert\Lambda^{s}v(\tau)\Vert_{L^{(\alpha+2,\infty)}}^{\gamma})\Vert
u(\tau)-v(\tau)\Vert_{L^{(\alpha+2,\infty)}}d\tau\\
&  \leq|b|CM^{\gamma}d(u,v)t^{1-\frac{n}{d}\frac{\alpha}{\alpha+2}%
-\gamma\delta}\int_{0}^{1}(1-s)^{-\frac{n}{d}\frac{\alpha}{\alpha+2}%
}s^{-(\beta+\gamma\delta)}ds\\
&  \leq|b|C_{2}M^{\gamma}d(u,v),
\end{align*}
where in the last inequality we used that $\gamma\delta+\beta<1$,
$\frac{n\alpha}{d(\alpha+2)}<1$, and
\begin{equation}
\frac{n}{d}\frac{\alpha}{\alpha+2}+\gamma\delta=1.\label{gaderel}
\end{equation}

For $J_{3}$, we use Lemmas \ref{lema2.1} and \ref{lemadeux} to
obtain
\begin{align*}
J_{3}\leq &  |a|Ct^{\delta}\int_{0}^{t}(t-\tau)^{-\frac{n}{d}\frac{\alpha
}{\alpha+2}}\Vert\Lambda^{s}(|u(\tau)|^{\alpha}u(\tau))\Vert_{L^{(\frac
{\alpha+2}{\alpha+1},\infty)}}d\tau\\
\leq &  |a|Ct^{\delta}\int_{0}^{t}(t-\tau)^{-\frac{n}{d}\frac{\alpha}%
{\alpha+2}}\Vert u(\tau)\Vert_{L^{(\alpha+2,\infty)}}^{\alpha}\Vert\Lambda
^{s}u(\tau))\Vert_{L^{(\alpha+2,\infty)}}d\tau
\end{align*}

In comparison with $J_{3}$, the handling of $J_{4}$ requires some care due to the
presence of the nonlocal operator $E$ and the fractional derivative
$\Lambda^{s}.$ For that, we recall the parameters $z=\frac{\alpha+2}{\alpha
}\gamma$ and $l=\frac{(\alpha+2)\gamma}{(\alpha+1)(\gamma-1)+1}$ and apply
Leibniz's rule in the setting of weak-$L^{p}$ (see \cite[Theorem 6.1]{CN}) to estimate%
\begin{equation}
\Vert\Lambda^{s}(E(|u|^{\gamma})u\Vert_{L^{(\frac{\alpha+2}{\alpha+1},\infty
)}}\leq\Vert\Lambda^{s}(E(|u|^{\gamma})\Vert_{L^{(l,\infty)}}\Vert
u\Vert_{L^{(z,\infty)}}+\Vert E(|u|^{\gamma})\Vert_{L^{(\frac{\alpha+2}%
{\alpha},\infty)}}\Vert\Lambda^{s}u\Vert_{L^{(\alpha+2,\infty)}}%
\label{ineq-IV-parcel-1}%
\end{equation}
Next, let $l_{1}=\frac{z}{\gamma-1}$ and note that $\frac{1}{l}=\frac{1}%
{l_{1}}+\frac{1}{\alpha+2}.$ Using \textbf{(H3)}, Lemma \ref{lematrois} and
then Remark \ref{remark1} in the R.H.S. of (\ref{ineq-IV-parcel-1}), we
obtain that
\begin{align}
\Vert\Lambda^{s}(E(|u|^{\gamma})u\Vert_{L^{(\frac{\alpha+2}{\alpha+1},\infty
)}}  & \leq C\Vert\Lambda^{s}(|u|^{\gamma})\Vert_{L^{(l,\infty)}}\Vert
u\Vert_{L^{(z,\infty)}}+C\Vert|u|^{\gamma}\Vert_{L^{(\frac{\alpha+2}{\alpha
},\infty)}}\Vert\Lambda^{s}u\Vert_{L^{(\alpha+2,\infty)}}\nonumber\\
& \leq C\Vert|u|^{\gamma-1}\Vert_{L^{(l_{1},\infty)}}\Vert\Lambda^{s}%
u\Vert_{L^{(\alpha+2,\infty)}}\Vert u\Vert_{L^{(z,\infty)}}+C\Vert
u\Vert_{L^{(z,\infty)}}^{\gamma}\Vert\Lambda^{s}u\Vert_{L^{(\alpha+2,\infty)}%
}\nonumber\\
& \leq C\Vert u\Vert_{L^{(z,\infty)}}^{\gamma-1}\Vert\Lambda^{s}%
u\Vert_{L^{(\alpha+2,\infty)}}\Vert u\Vert_{L^{(z,\infty)}}+C\Vert
u\Vert_{L^{(z,\infty)}}^{\gamma}\Vert\Lambda^{s}u\Vert_{L^{(\alpha+2,\infty)}%
}\nonumber\\
& \leq C\Vert\Lambda^{s}u\Vert_{L^{(\alpha+2,\infty)}}^{\gamma+1}%
\label{ineq-IV-parcel-2}%
\end{align}
Finally, Lemma \ref{lema2.1} and estimate \eqref{ineq-IV-parcel-2} yield%
\begin{align*}
J_{4}\leq &  |b|Ct^{\delta}\int_{0}^{t}(t-\tau)^{-\frac{n}{d}\frac{\alpha}%
{\alpha+2}}\Vert\Lambda^{s}(E(|u(\tau)|^{\gamma})u(\tau))\Vert_{L^{(\frac
{\alpha+2}{\alpha+1},\infty)}}d\tau\\
\leq &  |b|Ct^{\delta}\int_{0}^{t}(t-\tau)^{-\frac{n}{d}\frac{\alpha}%
{\alpha+2}}\Vert\Lambda^{s}u(\tau))\Vert_{L^{(\alpha+2,\infty)}}^{\gamma
+1}d\tau.
\end{align*}

Following the same ideas as before,  using \eqref{alberel}, \eqref{gaderel},  and the facts that
$ \beta\alpha+\delta<1,\ \delta(\gamma+1)<1$,
we have
\begin{align*}
 J_{3}+J_{4}\leq |a|C_3 M^{\alpha+1}+|b|C_4 M^{\gamma+1}.
\end{align*}
With the above estimates in hand we are able to prove existence of global solutions to \eqref{eqintds}. Indeed, suppose that $u_0\in\mathcal{I}_\rho$. The four estimates above combined with assumptions \eqref{a} and \eqref{a_1}  promptly allow us to show that $\Phi$ acts from $X_M$ to $X_M$. In addition, the estimates for $J_{1}$ and $J_{2}$ together with \eqref{b} implies the existence of a positive constant $K_0<1$ such that
\begin{align*}
d(\Phi u,\Phi v)\leq K_0 d(u,v),
\end{align*}
for any $u,v\in X_M$. The Banach fixed point theorem then gives us the desired result.
\end{proof}

\begin{remark}\label{rem1}
	Our results also holds true for NLS-like equations with nonlinearity having two powers and double nonlocal operators $E_1$ and $E_2$ satisfying \textbf{(H3)}. More precisely, in \eqref{NLS} we can consider the nonlinearity $a E_{1}(|u|^{\alpha}) u +bE_{2}(|u|^\gamma) u$. To do so, it is sufficient to estimate the terms $J_{1}$ and $J_{3}$ with the new term $E_{1}$ (see page 12) by following the arguments used to handle $J_{2}$ and $J_{4}$.
\end{remark}

\begin{remark} Here we give a sufficient condition to initial data satisfying the assumptions in Theorem \ref{globaltheorem1}. Let $\varphi \in \dot{H}^\sigma_{q,\infty}\cap \dot{H}^\tau_{r,\infty}$, where
$$
\sigma=-\frac{\beta d}{2}+\frac{n \alpha}{2(\alpha+2)},\quad \frac{1}{q}=\frac{\beta d}{2n}+\frac{1}{2},\quad \tau=s-\frac{\delta d}{2}+\frac{n\alpha}{2(\alpha+2)},\quad \frac{1}{r}=\frac{\delta d}{2n}+\frac{1}{2},
$$
with $\beta,\delta,\alpha,\gamma,s$ satisfying the assumptions in Theorem \ref{globaltheorem1}. Noting that

$$
\sigma=n\left(\frac{1}{q'}-\frac{1}{\alpha+2}\right) \qquad \mbox{and} \qquad \tau-s=n\left(\frac{1}{r'}-\frac{1}{\alpha+2}\right),
$$
and using Lemmas \ref{lemaun} and \ref{lema2.1}, we estimate
\begin{align*}
\|U(t)\varphi\|_{L^{(\alpha+2,\infty)}}=\|\Lambda^{-\sigma} U(t)\Lambda^{\sigma}\varphi\|_{L^{(\alpha+2,\infty)}}\leq C \|U(t)\Lambda^{\sigma}\varphi\|_{L^{(q',\infty)}}\leq C|t|^{-\beta}\|\Lambda^{\sigma}\varphi\|_{L^{(q,\infty)}}.
\end{align*}
Analogously,
\begin{align*}
\|\Lambda^sU(t)\varphi\|_{L^{(\alpha+2,\infty)}}=\|\Lambda^{s-\tau} U(t)\Lambda^{\tau}\varphi\|_{L^{(\alpha+2,\infty)}}\leq C \|U(t)\Lambda^{\tau}\varphi\|_{L^{(r',\infty)}}\leq  C|t|^{-\delta}\|\Lambda^{\tau}\varphi\|_{L^{(r,\infty)}}.
\end{align*}

It follows that $U(t)\varphi \in X_\rho$ for some $\rho>0$. In particular, the assumptions in Theorem \ref{globaltheorem1} hold provided that $u_0 \in \dot{H}^\sigma_{q,\infty}\cap \dot{H}^\tau_{r,\infty}$ and $\rho$ and $M$ are small enough.

\end{remark}

\begin{remark}[Well-posedness]\label{rem2}
	Under the assumptions of Theorem \ref{globaltheorem1}, if $u_0$ and $v_0$ are two tempered distributions in $X_\rho$ and if $u$ and $v$ are the corresponding solutions then we easily see that there exists a constant $C>0$ such that
	$$
	d(u,v)\leq Cd(U(t)u_0,U(t)v_0),
	$$
which shows the continuous dependence of solutions with respect to initial data. Thus, we have in fact obtained a well-posedness result in the sense of Hadamard in our setting.
\end{remark}

\begin{remark}
A few words of explanation concerning our assumptions in Theorem \ref{globaltheorem1} are in order. For $r>0$, let $\{r\}$ denotes the smallest integer bigger than or equal to $r$.  Instead of assuming $0<s<1$, in \cite{W} it was assumed
\begin{equation}\label{srel}
\{s\}<\alpha+1.
\end{equation}
In particular, when $s<1$, \eqref{srel} is equivalent to $\alpha>0$, which brings no additional assumption on $\alpha$. Hence our assumption is more restrictive than the one in \cite{W}. On the other hand, the assumption $s<1$ is equivalent to
\begin{equation}\label{srel1}
\alpha>\frac{(n-2)\gamma}{\gamma+n}.
\end{equation}
It is clear that if $n=1$ or $n=2$, \eqref{srel1} is always true. So, our assumption makes sense only in dimension $n\geq3$ in which case we are indeed assuming that the relation between $\alpha$ and $\gamma$ satisfies
$$
\frac{(n-2)\gamma}{\gamma+n}<\alpha<\gamma.
$$
\end{remark}

In the sequel we will study some properties of the global solution obtained in Theorem \ref{globaltheorem1}. The first one concerns scattering in weak Lebesgue spaces.

\begin{theorem}[Scattering]\label{scatering}
	Suppose that the assumptions in Theorem \ref{globaltheorem1} hold and
let $u$ be the corresponding global solution with initial data $u_0$.
Then, there exists $u_+\in \mathcal{I}_{\tilde\rho}$, for some $\tilde\rho>0$,  such that

\begin{equation}\label{a10.2}
\|u(t)-U(t)u_+\|_{L^{(\alpha+2,\infty)}}\leq
Ct^{-\beta}\Big(\|u\|_\beta^{\alpha+1}+\|u\|^\gamma_{\delta,s}\|u\|_\beta\Big), \quad t>0.
\end{equation}
and
\begin{equation}\label{a10.22}
\|\Lambda^s\left(u(t)-U(t)u_+\right)\|_{L^{(\alpha+2,\infty)}}\leq
Ct^{-\delta}\Big(\|u\|_\beta^{\alpha}\|u\|_{\delta,s}+\|u\|^{\gamma+1}_{\delta,s}\Big), \quad t>0.
\end{equation}
In particular
$$
\lim_{t\rightarrow\infty}\Big(\|u(t)-U(t)u_+\|_{L^{(\alpha+2,\infty)}}+ \|\Lambda^s\left(u(t)-U(t)u_+\right)\|_{L^{(\alpha+2,\infty)}}\Big)=0.
$$
\end{theorem}
\begin{proof}
To simplify notation let us write $F(s)=a(|u|^\alpha u)(s)+b(E(|u|^\gamma) u)(s)$. From \eqref{eqintds}, we
have, for $t>1$,
\begin{equation}\label{a11}
U(-t)u(t)=u_0+i\int_0^1U(-s)F(s)ds+i\int_1^tU(-s)F(s)ds.
\end{equation}
Let us show that the last integral on the right-hand side of \eqref{a11} is convergent as $t\to\infty$. In fact, as in the proof of Theorem \ref{globaltheorem1},
\begin{equation*}
\begin{split}
\int_1^t\|U(-s)F(s)\|_{L^{(\alpha+2,\infty)}}ds&\leq
C\|u\|_\beta^{\alpha+1}\int_1^ts^{-{\frac{n}{d}\frac{\alpha}{\alpha+2}}}s^{-\beta(\alpha+1)}ds\\ &\quad \quad+C\|u\|_{\delta,s}^\gamma\|u\|_\beta \int_1^ts^{-{\frac{n}{d}\frac{\alpha}{\alpha+2}}}s^{-(\gamma\delta+\beta)}ds\\
&\leq CM^{\alpha+1}\int_1^ts^{-\left(1+\beta\right)}ds+ CM^{\gamma+1}\int_1^ts^{-\left(1+\beta\right)}ds\\
&\leq C(M)(1-t^{-\beta}),
\end{split}
\end{equation*}
where in the second inequality we have used \eqref{alberel} and \eqref{gaderel}. This implies
that the distribution
$$
u_+:=u_0+i\int_0^\infty U(-s)F(s)ds
$$
is well-defined. Note that
$$
U(t)u_+=U(t)u_0+i\int_0^\infty U(t-s)F(s)ds.
$$
We are going to show that $u_+\in \mathcal{I}_{\tilde\rho}$, for some $\tilde\rho>0$. First we claim that $\|u_+\|_{\beta}<\infty$. For that, it is sufficient to show that
\begin{equation}\label{interel}
\int_0^\infty\|U(t-s)F(s)\|_{L^{(\alpha+2,\infty)}}\leq Ct^{-\beta}, \text{ for all } t>0.
\end{equation}
To establish \eqref{interel} we split
\begin{equation}\label{a12}
\int_0^\infty U(t-s)F(s)ds=\int_0^tU(t-s)F(s)ds+\int_t^\infty U(t-s)F(s)ds.
\end{equation}
The first integral on the right-hand side of \eqref{a12} can be estimated as
in the proof of Theorem \ref{globaltheorem1}. So that,
\begin{equation}\label{a13}
\int_0^t\|U(t-s)F(s)\|_{L^{(\alpha+2,\infty)}}ds\leq
Ct^{-\beta}\Big(\|u\|_\beta^{\alpha+1}+\|u\|^\gamma_{\delta,s}\|u\|_\beta\Big)\leq C(M^{\alpha+1}+M^{\gamma+1})t^{-\beta}.
\end{equation}
For the second integral, by using \eqref{alberel} and \eqref{gaderel}, we deduce
\begin{equation}\label{a14}
\begin{split}
\int_t^\infty\|U(t-s)F(s)\|_{L^{(\alpha+2,\infty)}}ds&\leq
Ct^{-\beta}M^{\alpha+1}\int_1^\infty
(s-1)^{-{\frac{n}{d}\frac{\alpha}{\alpha+2}}}s^{-(\alpha\beta+\beta)}ds\\
& \quad + Ct^{-\beta}M^{\gamma+1}\int_1^\infty
(s-1)^{-{\frac{n}{d}\frac{\alpha}{\alpha+2}}}s^{-(\gamma\delta+\beta)}ds\\
& =Ct^{-\beta} I_1+Ct^{-\beta}I_2,
\end{split}
\end{equation}
where the integrals $I_1$ and $I_2$ are finite. Using \eqref{a13} and \eqref{a14}, we obtain \eqref{interel}. From our calculations above, it is also clear that \eqref{a10.2} holds.

 By following the arguments above and using the estimates for $J_{3}$ and $J_{4}$, it is not difficult to see that $\|u_+\|_{\delta, s}<\infty$ and that \eqref{a10.22} also holds. The proof of the theorem is thus completed.
\end{proof}

In next result, we investigate suitable conditions on the initial data so that solutions present decay faster than those in Theorem \ref{globaltheorem1}.

\begin{theorem}\label{globaltheorem2}
Under the hypotheses of Theorem \ref{globaltheorem1}.
\begin{itemize}
\item[(i)]  Assume that $u_0$ satisfies $\| U(t)u_0\|_{\mu}<\infty$ for some $\mu \geq 0$ with $\alpha \beta +\mu <1.$ Assume also that there exist positive constants $C_5$ and $C_6$ such that
\begin{align}\label{ab}
|a|C_5M^\alpha +|b|C_6M^\gamma <1,
\end{align}
where $C_5$ and $C_6$ are constants explicitly obtained in the calculations below.
Then, the solution $u$ given by Theorem \ref{globaltheorem1} verifies the decay property $\| u\|_{\mu}<\infty$.

\item[(ii)]  Assume that $u_0$ satisfies $\|U(t) u_0\|_{\nu,s}<\infty$ for some $\nu \geq 0$ with $\alpha \beta +\nu <1.$ Assume also that there exist positive constants $C_7$ and $C_8$ such that
$$
|a|C_7M^\alpha +|b|C_8M^\gamma <1,
$$
where $C_7$ and $C_8$ are constant explicitly obtained in the calculations below. Then, the solution $u$ given by Theorem \ref{globaltheorem1} satisfies the property $\| u\|_{\nu,s}<\infty$.
\end{itemize}
\end{theorem}
\begin{proof}

Working as in the estimates for $J_{1}$ and $J_{2}$ above, we can estimate
\begin{align*}
J_{5}:= &|a|t^\mu\int_0^t \|U(t-\tau)(|u|^\alpha)(\tau)\|_{L^{(\alpha+2,\infty)} }d\tau\\
 \leq & |a|C M^\alpha t^\mu\int_0^t(t-\tau)^{-\frac{n}{d}(\frac{\alpha}{\alpha+2})}\tau^{-\beta \alpha-\mu}d\tau\sup_{t>0}t^\mu\|u(t)\|_{L^{(\alpha+2,\infty)}}\\
 \leq &
|a| C_5M^\alpha \|u\|_{\mu},
\end{align*}
where in the last inequality we used \eqref{alberel} and the facts that $\frac{n\alpha}{d(\alpha+2)}<1$ and $ \beta\alpha+\mu <1$.

Now using the same ideas,
\begin{align*}
 J_{6}:=&|b|t^\mu \int_0^t \|U(t-\tau)[E(|u|^\gamma) u (\tau)]\|_{L^{(\alpha+2,\infty)} }d\tau\\
  \leq & |b|C M^\gamma t^\mu\int_0^t(t-\tau)^{-\frac{n}{d}(\frac{\alpha}{\alpha+2})}\tau^{-\delta \gamma-\mu}d\tau \sup_{t>0}t^\mu\|u(t)\|_{L^{(\alpha+2,\infty)}}\\
 \leq &
|b| C_6M^\gamma \|u\|_{\mu},
\end{align*}
where in the last inequality we used that $ \frac{n\alpha}{d(\alpha+2)}<1,\ \delta\gamma+\mu <1$ and \eqref{gaderel}. Note that $\alpha\beta=\gamma\delta$, so the assumption $\alpha\beta+\mu<1$ also gives $\gamma\delta+\mu<1$.

We will reapply the contraction-mapping argument in order to get the desired property. Since $\| U(t)\varphi\|_{\mu}<\infty$, there exists $\sigma>0$ such that $\| U(t)\varphi\|_{\mu}\leq\sigma$. From hypothesis \eqref{ab} we have that $1-(|a|C_5M^\alpha +|b|C_6M^\gamma)>0$. Hence, by  choosing  $K>0$ such that
$$
K\geq \frac{\sigma}{1-(|a|C_5M^\alpha +|b|C_6M^\gamma)},
$$
we deduce that
$$\sigma+|a|C_5M^\alpha K +|b|C_6M^\gamma K \leq K.$$
Now consider the following subspace $Y_{M,K}\subset X_M$:
$$Y_{M,K}=\{w\in X_M; \| w\|_{\mu}\leq K\}.$$
Observe that $(Y_{M,K},d)$ is a nonempty complete metric space, with $d$ as in Definition \ref{def1}.

Let $\Phi$ be the integral operator defined in \eqref{Z1}. Let us show that $\Phi$ maps $Y_{M,K}$
into itself and $\Phi:Y_{M,K}\rightarrow
Y_{M,K}$ is a contraction.
Suppose that $u \in Y_{M,K}$.
Estimates for $J_{5}$ and $J_{6}$ yield
$$\| \Phi(u)\|_{\mu}\leq \sigma+|a|C_5M^\alpha K +|b|C_6M^\gamma K \leq K,$$
proving that $\Phi(Y_{M,K})\subset Y_{M,K}$. Since the distance in $Y_{M,K}$ is that in $X_M$ and we already proved that $\Phi$ is a contraction on $X_M$ (see proof of Theorem \ref{globaltheorem1}) we have that $\Phi$ is also a contraction on $Y_{M,K}$. This implies that the solution $u$ in $X_M$ is also in $Y_{M,K}$, which means that $\|u\|_\mu<\infty$.

Now we turn to item $(ii)$. Suppose that  $\| U(t)\varphi\|_{\nu,s}\leq\sigma$ and choose $K>0$ such that
$$\sigma+|a|C_7M^\alpha K +|b|C_8M^\gamma K \leq K.$$
By defining $Z_{M,K}\subset X_M$ as
$$Z_{M,K}=\{w\in X_M; \| w\|_{\nu,s}\leq K\},$$
we see that $(Z_{M,K},d)$ is a nonempty complete metric space with the metric  $d$ as in Definition \ref{def1}. Let us show that $\Phi$ maps $Z_{M,K}$
into itself and $\Phi:Z_{M,K}\rightarrow
Z_{M,K}$ is a contraction.
By assuming $u \in Z_{M,K}$ and slightly adapting the estimate for $J_{3}$ in Theorem \ref{globaltheorem1}, we easily arrive at
\begin{align*}
 J_{7} &:=|a|t^\nu\int_0^t\|\Lambda^s U(t-\tau)(|u(\tau)|^\alpha u(\tau))\|_{L^{({\alpha+2},\infty)}}d\tau\\
 &\leq |a| C t^\nu\int_0^t(t-\tau)^{-\frac{n}{d}(\frac{\alpha}{\alpha+2})}\|u(\tau) \|^\alpha_{L^{(\alpha+2,\infty)}}\|\Lambda^s u(\tau))\|_{L^{(\alpha+2,\infty)}}d\tau\\
 &\leq |a|C t^\nu M^\alpha K\int_0^t(t-\tau)^{-\frac{n}{d}(\frac{\alpha}{\alpha+2})}\tau^{-\beta\alpha-\nu}d\tau\leq |a|C_7M^\alpha K,
\end{align*}
where  we used \eqref{alberel}, $ \frac{n\alpha}{d(\alpha+2)}<1$, and $\beta\alpha+\nu <1$.

Now we proceed as in the estimate for $J_{4}$ to get
\begin{align*}
 J_{8} &:=|b|t^\nu\int_0^t\|\Lambda^s U(t-\tau)[E(|u(\tau)|^\gamma) u(\tau)]\|_{L^{({\alpha+2},\infty)}}d\tau\\
 &\leq |b|C t^\nu\int_0^t(t-\tau)^{-\frac{n}{d}(\frac{\alpha}{\alpha+2})}\|\Lambda^s u(\tau))\|^{\gamma+1}_{L^{(\alpha+2,\infty)}}d\tau\\
 &\leq  |b|Ct^{\nu}M^\alpha K\int_0^t(t-\tau)^{-\frac{n}{d}(\frac{\alpha}{\alpha+2})}\tau^{-\delta\gamma-\nu}d\tau\leq |b|C_8M^\gamma K,
\end{align*}
where  we used \eqref{gaderel}, $ \frac{n\alpha}{d(\alpha+2)}<1$, and $\delta\gamma+\nu <1$.

Using the estimates for $J_{7}$ and $J_{8}$, it follows that
\begin{align*}
\|\Phi(u)\|_{\nu,s}\leq \sigma
 +|a|C_7M^\alpha K+|b|C_8M^\gamma K\leq K,
\end{align*}
proving that $\Phi(Z_{M,K})\subset Z_{M,K}$. The conclusion then follows as in the first part and we are done.
\end{proof}

From another point of view, the solution $u$ in Theorem \ref{globaltheorem1} and Theorem \ref{globaltheorem2} $(i)$ satisfies
$\Vert u(\cdot,t)\Vert_{L^{(\alpha+2,\infty)}}=\mathcal{O}(t^{-\mu})$ provided that
$\Vert U(t)u_{0}\Vert_{L^{(\alpha+2,\infty)}}=\mathcal{O}(t^{-\mu})$ as $t\rightarrow
\infty$, for $\mu=\beta$ and $\mu(\alpha+1)<1$, respectively. In the sequel we provide a criterion for solutions to be
asymptotically stable,  which, in particular, assures that we can replace $\mathcal{O}(t^{-\mu})$
by $o(t^{-\mu})$ in the last two equalities.

\begin{theorem}[Asymptotic Stability]\label{asymest} Assume the hypotheses of Theorem
\ref{globaltheorem1} and that $u_{0},v_{0}\in\mathcal{I}_{\rho}$. For some
$\mu\geq\beta$ with $\alpha\beta+\mu<1,$ suppose that $\left\Vert
U(t)u_{0}\right\Vert _{\mu}<\infty$, $\left\Vert U(t)v_{0}\right\Vert _{\mu
}<\infty$ and
\[
\lim_{t\rightarrow\infty}t^{\mu}\Vert U(t)(u_{0}-v_{0})\Vert_{L^{(\alpha
+2,\infty)}}=0.
\]
For $\mu>\beta,$ assume that $|a|C_{5}M^{\alpha}+|b|C_{6}M^{\gamma}<1$. For
$\mu=\beta$, we have condition \eqref{b}. Let $u$ and $v$ be the solutions of
\eqref{eqintds} with initial values $u_{0}$ and $v_{0}$, respectively, given
by Theorem \ref{globaltheorem1}. Then
\[
\lim_{t\rightarrow\infty}t^{\mu}\Vert u(t)-v(t)\Vert_{L^{(\alpha+2,\infty)}%
}=0.
\]

\end{theorem}
\begin{proof}
 For $\mu\geq\beta$, we can estimate
\begin{align*}
t^{\mu}\Vert u(t)-v(t)\Vert_{L^{(\alpha+2,\infty)}} &  \leq t^{\mu}\Vert
U(t)(u_{0}-v_{0})\Vert_{L^{(\alpha+2,\infty)}}\\
&  \quad+|a|CM^{\alpha}t^{\mu}\int_{0}^{t}(t-\tau)^{-\frac{n}{d}(\frac{\alpha
}{\alpha+2})}\tau^{-\beta\alpha-\mu}\tau^{\mu}\Vert u(\tau)-v(\tau
)\Vert_{L^{(\alpha+2,\infty)}}d\tau\\
&  \quad+|b|CM^{\gamma}t^{\mu}\int_{0}^{t}(t-\tau)^{-\frac{n}{d}(\frac{\alpha
}{\alpha+2})}\tau^{-\delta\gamma-\mu}\tau^{\mu}\Vert u(\tau)-v(\tau
)\Vert_{L^{(\alpha+2,\infty)}}d\tau
\end{align*}

Making the change of variables $\tau=tz,$ we arrive at
\begin{align}
t^{\mu}\Vert u(t)-v(t)\Vert_{L^{(\alpha+2,\infty)}} &  \leq t^{\mu}\Vert
U(t)(u_{0}-v_{0})\Vert_{L^{(\alpha+2,\infty)}}\nonumber\\
&  \quad+|a|CM^{\alpha}\int_{0}^{1}(1-z)^{-\frac{n}{d}(\frac{\alpha}{\alpha
+2})}z^{-\beta\alpha-\mu}(tz)^{\mu}\Vert u(tz)-v(tz)\Vert_{L^{(\alpha
+2,\infty)}}dz\nonumber\\
&  \quad+|b|CM^{\gamma}\int_{0}^{1}(1-z)^{-\frac{n}{d}(\frac{\alpha}{\alpha
+2})}z^{-\delta\gamma-\mu}(tz)^{\mu}\Vert u(tz)-v(tz)\Vert_{L^{(\alpha
+2,\infty)}}dz\label{aux-stab-1}%
\end{align}
Denote
\[
A=\limsup_{t\rightarrow\infty}t^{\mu}\Vert u(t)-v(t)\Vert_{L^{(\alpha
+2,\infty)}}.
\]
For $\mu>\beta$ and $\mu=\beta$, we have that $A<\infty$ due to Theorems \ref{globaltheorem2} and \ref{globaltheorem1}, respectively. Computing
\textit{limsup} in both sides of (\ref{aux-stab-1}) and using Dominated
Convergence Theorem, we obtain
\begin{align}
A  & \leq0+|a|CM^{\alpha}\int_{0}^{1}(1-z)^{-\frac{n}{d}(\frac{\alpha}%
{\alpha+2})}z^{-\beta\alpha-\mu}dz\text{ }A+|b|CM^{\gamma}\int_{0}%
^{1}(1-z)^{-\frac{n}{d}(\frac{\alpha}{\alpha+2})}z^{-\delta\gamma-\mu
}dzA\nonumber\\
& =RA,\label{aux-stab-2}%
\end{align}
where $R=|a|C_{1}M^{\alpha}+|b|C_{2}M^{\gamma}<1$ for $\mu=\beta$ and
$R=|a|C_{5}M^{\alpha}+|b|C_{6}M^{\gamma}<1$ for $\mu>\beta.$  Thus, it follows
that $A=0,$ which gives the desired conclusion.
\end{proof}

Equation \eqref{NLS} has no scaling $u(x,t)\mapsto \lambda^{m} u( \lambda x, \lambda^2 t)$, for any $m\in\R$. This fact prevents the existence of self-similar solutions to \eqref{NLS}. Alternatively, we will prove that \eqref{NLS} admits a class of asymptotically self-similar solutions with respect to the scaling of \eqref{singlenls}. Resembling results in the $L^p$-setting can be found in \cite{W}.

For that matter, first note that Theorem \ref{globaltheorem1} with $b=0$ gives global mild solutions in
weak-$L^{p}$ spaces for \eqref{singlenls} (see also \cite{BP, SFR}), i.e., solutions $v$ of the integral equation

\begin{equation}
v(t)=U(t)v_{0}+i\int_{0}^{t}U(t-s)(a|v|^{\alpha}v)(s)ds,\label{eqintdsa}%
\end{equation}
satisfying $\|v\|_{\beta}\leq M$. Moreover, with a slight modification in the proof, we  only need to assume $\|U(t)v_0\|_{\beta}\leq\rho$ instead of $v_0\in\mathcal{I}_{\rho}$.

In the next theorem, we compare the mild solutions of \eqref{NLS} and
\eqref{singlenls}. In fact, we are going to prove that solutions of
\eqref{singlenls} attract those of \eqref{NLS} as $t\rightarrow\infty,$
depending on a suitable condition for the difference of the initial values
$\psi=u_{0}-v_{0}$.

\begin{theorem}
\label{sch-asymp} Under the hypotheses of Theorem \ref{globaltheorem1}. Let
$u$ be the corresponding solution of \eqref{eqintds} with initial value
$u_{0}\in\mathcal{I}_{\rho}$ given by Theorem \ref{globaltheorem1}. Let $v$ be
the solution of \eqref{eqintdsa} (i.e., \eqref{eqintds} with $b=0$) with
initial value $v_{0}$, such that $\|U(t)v_0\|_{\beta}\leq\rho$, also given by Theorem
\ref{globaltheorem1}. Suppose further that $u_{0}$ satisfies the hypotheses of
part (ii) of Theorem \ref{globaltheorem2} with some $\nu>\delta$ such that
$\gamma\nu+\beta<1$. Then, we have that

\begin{equation}
\lim_{t\rightarrow\infty}t^{\beta}\Vert u(t)-v(t)\Vert_{L^{(\alpha+2,\infty)}%
}=0,\label{stab-100}
\end{equation}

provided that
\begin{equation}
\lim_{t\rightarrow\infty}t^{\beta}\Vert U(t)(u_{0}-v_{0})\Vert_{L^{({\alpha
+2},\infty)}}=0.\label{cond-stab-100}
\end{equation}

\end{theorem}

\begin{proof}
First note that

\begin{align*}
t^{\beta}\Vert u(t)-v(t)\Vert_{L^{({\alpha+2},\infty)}} &  \leq t^{\beta}\Vert
U(t)(u_{0}-v_{0})\Vert_{L^{({\alpha+2},\infty)}}\\
&  \quad+|a|t^{\beta}\Vert\int_{0}^{t}U(t-\tau)(|u|^{\alpha}u-|v|^{\alpha
}v)(\tau)d\tau\Vert_{L^{(\alpha+2,\infty)}}\Vert_{L^{(\alpha+2,\infty)}}\\
&  \quad+|b|t^{\beta}\Vert\int_{0}^{t}U(t-\tau)[E(|u|^{\gamma})u](\tau
)d\tau\Vert_{L^{(\alpha+2,\infty)}}%
\end{align*}
Proceeding as in the estimates for $J_{1}$ and $J_{2}$ above, we obtain

\begin{align}
t^{\beta}\Vert u(t)-v(t)\Vert_{L^{({\alpha+2},\infty)}} &  \leq t^{\beta}\Vert
U(t)(u_{0}-v_{0})\Vert_{L^{({\alpha+2},\infty)}}\nonumber\\
&  +|a|CM^{\alpha}t^{\beta}\int_{0}^{t}(t-\tau)^{-\frac{n}{d}(\frac{\alpha
}{\alpha+2})}\tau^{-\beta\alpha-\beta}\tau^{\beta}\Vert u(\tau)-v(\tau
)\Vert_{L^{(\alpha+2,\infty)}}d\tau\nonumber\\
&  +|b|CM^{\gamma}t^{\beta}\int_{0}^{t}(t-\tau)^{-\frac{n}{d}(\frac{\alpha
}{\alpha+2})}\Vert\Lambda^{s}u(\tau))\Vert_{L^{(\alpha+2,\infty)}}^{\gamma
}\Vert u(\tau)\Vert_{L^{(\alpha+2,\infty)}}d\tau\nonumber\\
&  :=A_{1}(t)+A_{2}(t)+A_{3}(t).\label{aux-stab-3}%
\end{align}
From Theorem \ref{globaltheorem2} part (ii), we have that $\Vert u\Vert
_{\nu,s}=$ $\sup_{\tau>0}\tau^{\nu}\Vert\Lambda^{s}u(\tau)\Vert_{L^{(\alpha
+2,\infty)}}<\infty$ and then
\begin{align}
A_{3}(t) &  \leq C|b|t^{\beta}\int_{0}^{t}(t-\tau)^{-\frac{n}{d}(\frac{\alpha
}{\alpha+2})}\tau^{-\gamma\nu-\beta}d\tau\Vert u\Vert_{\nu,s}^{\gamma}\Vert
u\Vert_{\beta}\nonumber\\
&  =t^{{-\frac{n}{d}(\frac{\alpha}{\alpha+2})}+1-\gamma\delta-\gamma
(\nu-\delta)}C|b|\int_{0}^{1}(1-z)^{{-\frac{n}{d}(\frac{\alpha}{\alpha+2}%
)}}\tau^{-\gamma\nu-\beta}d\tau\Vert u\Vert_{\nu,s}^{\gamma}\Vert
u\Vert_{\beta}\nonumber\\
&  =Ct^{-\gamma(\nu-\delta)}\Vert u\Vert_{\nu,s}^{\gamma}\Vert u\Vert_{\beta
}\rightarrow0,\text{ as }t\rightarrow\infty.\label{aux-stab-4}%
\end{align}
Taking $\mathcal{H}=\limsup_{t\rightarrow\infty}t^{\beta}\Vert u(t)-v(t)\Vert
_{L^{(\alpha+2,\infty)}}$, working in the same spirit of the proof of
Theorem \ref{asymest} and using (\ref{aux-stab-3})-(\ref{aux-stab-4}), we
obtain
\[
\mathcal{H}\leq0+|a|C_{1}M^{\alpha}\mathcal{H}+0,
\]
which gives (\ref{stab-100}), because $0\leq\mathcal{H}<\infty$ and
$|a|C_{1}M^{\alpha}<1.$
\end{proof}

We finish by showing the existence of mild solutions of \eqref{NLS} that are asymptotically self-similar at infinite, with respect to the scaling of the single-power Schr\"odinger equation \eqref{singlenls}.

\begin{corollary}
(Asymptotic self-similarity) In addition to the hypotheses of Theorem
\ref{sch-asymp}, assume that $v_{0}$ is a homogeneous distribution of degree
$-d/\alpha$ and that $u_{0}=v_{0}+\omega$ with $\omega$ satisfying (\ref{cond-stab-100}), e.g., $\omega\in L^{(\frac{\alpha+2}{\alpha+1},\infty)}$. Let $v$ be the self-similar solution
of \eqref{eqintdsa} corresponding to the initial value $v_{0}.$ Then, the solution $u$ of \eqref{eqintds} with initial value $u_{0}$ satisfies
\begin{equation}
\lim_{t\rightarrow\infty}t^{\beta}\left\Vert u(t)-v(t)\right\Vert
_{L^{(\alpha+2,\infty)}}=0.\label{aux-con-1000}%
\end{equation}
In other words, one obtains a class of solutions of \eqref{eqintds} that are
attracted in the sense of \eqref{aux-con-1000} to the self-similar solution
$v$ of \eqref{eqintdsa}.
\end{corollary}

\begin{proof} For $\omega\in L^{(\frac{\alpha+2}{\alpha+1},\infty)},$ using
Lemma \ref{lema2.1} and noting that $\beta<\frac{n\alpha}{d(\alpha+2)}$, we
have that
\[
0\leq t^{\beta}\left\Vert U(t)\omega\right\Vert _{L^{(\alpha+2,\infty)}}\leq
Ct^{\beta-\frac{n\alpha}{d(\alpha+2)}}\left\Vert \omega\right\Vert
_{L^{(\frac{\alpha+2}{\alpha+1},\infty)}}\rightarrow0,\text{ as }%
t\rightarrow\infty,
\]
and then $\omega=u_{0}-v_{0}$ satisfies (\ref{cond-stab-100}). If this condition is verified, then
(\ref{aux-con-1000}) follows from Theorem \ref{sch-asymp}. Furthermore, notice
that $v$ is self-similar because $v_{0}$ is homogeneous of degree $-d/\alpha$ (see, for instance,
\cite{CW1, BP, SFR}). The proof of the corollary is thus completed.

\end{proof}

\section*{Acknowledgment}
V.B. was partially supported by FCT project PTDC/MAT-PUR/28177/2017, with national funds, and by CMUP (UID/MAT/00144/2019), which is funded by FCT with national (MCTES) and European structural funds through the programs FEDER, under the partnership agreement PT2020. L.C.F.F. was partially supported by CNPq and FAPESP, Brazil. A.P. was partially supported by CNPq grants 402849/2016-7 and 303098/2016-3.


\begin{thebibliography}{article}
\addcontentsline{toc}{chapter}{Bibliography}

 \bibitem[AS]{AS} P. Antonelli, C. Sparber, Existence of solitary waves in dipolar quantum gases, Physica D 240 (2011), 426-431

 \bibitem[BEE]{BEE} C. Babaoglu, A. Eden, S. Erbay, Global existence and nonexistence results for a generalized Davey-Stewartson system, J. Phys. A: Math. Gen. 37 (2004), 11531-11546.

  \bibitem[BaE]{BaE} C. Babaoglu, S. Erbay, Two-dimensional wave packets in an elastic solid with
  couple stresses, Int. J. Nonlinear Mech. \textbf{39} (2004), 941-949.


 \bibitem[B]{B} V. Barros. The Davey Stewartson system in Weak $L^p$ Spaces
  Differential Integral Equations \textbf{25}  (2012), 883-898.
  
   \bibitem[BKS]{BKS} M. Ben-Artzi, H. Koch, J.-C. Saut, Dispersion estimates for fourth order Schrödinger equations,   C. R. Acad. Sci. Paris S\'er. I Math. \textbf{330}  (2000), 87-92.

 \bibitem[BeL]{BeL} J. Bergh, J. Lofstrom, Interpolation Spaces. An introduction, Springer-Verlag, Berlin-Heidelberg-New York, 1976.


 \bibitem[BP]{BP} V. Barros, A. Pastor, Infinite Energy solutions for Schrodinger-type equations with a nonlocal term, Advances in Differential Equations \textbf{18} no. 7-8 (2013), 769-796.



\bibitem[Ca]{Ca} T. Cazenave, Semilinear Schr\"odinger Equations, Courant Lectures Notes in Mathematics, Vol. 10, American Mathematical Society, Providence, 2003.

 \bibitem[CW1]{CW1} T. Cazenave, F.B. Weissler, Asymptotically self-similar global solutions of the nonlinear
 Schr\"odinger and heat equations, Math. Z. \textbf{228} (1998), 83-120.

\bibitem[CW2]{CW2} T. Cazenave, F.B. Weissler, Scattering theory and self-similar solutions for the nonlinear
Schr\"odinger equation, SIAM J. Math. Anal. \textbf{31} (2000), 625-650.






\bibitem[CW3]{CW3} T. Cazenave, F.B. Weissler, More self-similar solutions of the nonlinear
 Schr\"odinger equations, NoDEA Nonlinear Differential Equations Appl. \textbf{5} (1998), 355-365.

\bibitem[CMZ]{CMZ} X. Cheng, C. Miao, L. Zhao, Global well-posedness and scattering for nonlinear Schr\"odinger equations with combined nonlinearities in the radial case, J. Differential Equations \textbf{261} (2016), 1881-1934.



\bibitem[ChW]{ChW} F.M. Christ, M.I. Weinstein, Dispersions of small amplitude solutions of the generalizedKorteweg-de Vries equation, Journal of Functional Aalalysis. \textbf{100} (1991), 87-109.

\bibitem[Ci]{Ci} R. A. Cipolatti, On the existence of standing waves for a Davey-Stewartson system,
Comm. Partial Differential Equations \textbf{17} (1992), 967-988.

\bibitem[CN]{CN} D. Cruz-Uribe, V. Naibo, Kato-Ponce inequalities on weighted and variable Lebesgue spaces, Differential Integral Equations \textbf{29} (9-10) (2016), 801-836.

\bibitem[DS]{DS} A. Davey, K. Stewartson, On three dimensional packets of surface waves, Proc. Roy. London Soc.
A \textbf{338} (1974), 101-110.

\bibitem[Du]{Du} J.M. Dudley, C. Finot, D.J. Richardson, G. Millot, Self-similarity in ultrafast nonlinear optics, Nature Phys. 3 (9) (2007) 597-603.

\bibitem[FV-R]{FV-R} L.C.F. Ferreira, E.J. Villamizar-Roa, Self-similarity and asymptotic stability for coupled nonlinear Schr\"odinger equations in high dimensions. Phys. D 241 (5) (2012), 534-542.

\bibitem[FKTD]{FKTD} M.E. Fermann, V.I. Kruglov, B.C. Thomsen, J.M. Dudley, J.D. Harvey, Self-similar propagation and amplification of parabolic pulses in optical fibers, Phys. Rev. Lett. 84 (2000) 6010-6013.

\bibitem[Fi]{Fi} G. Fibich, The Nonlinear Schr\"odinger Equations, Singular Solutions and Optical Collapse, Applied Mathematical Sciences, Vol. 192, Springer, 2015.

 \bibitem[G]{G} R.H.J. Grimshaw, The modulation of an internal gravity-wave packet and the resonance with the
  mean motion, Stud. Appl. Math. \textbf{56}  (1977), 241-266.

\bibitem[Gr]{Gr} L. Grafakos, Classical Fourier Analysis, Second Edition, Springer-Verlag, New York, 2008.

\bibitem[Gr1]{Gr1} L. Grafakos, Modern Fourier Analysis, Second Edition, Springer-Verlag, New York, 2009.

 \bibitem[GS]{GS} J.M. Ghidaglia, J.C. Saut, On the initial problem for the Davey-Stewartson systems, Nonlinearity
\textbf{3} (1990), 475-506.

  \bibitem[GS1]{GS1} J.M. Ghidaglia, J.C. Saut, Nonelliptic Schr\"odinger equations, J. Nonlinear Sci.
\textbf{3}  (1993), 169-195.


  \bibitem[GSh]{GSh} B.L. Guo, C.X. Shen, Almost conservations law and global rough solutions to a linear
Davey-Stewartson equation, J. Math. Anal. Appl. \textbf{318}  (2006),
365-379.



\bibitem[H1]{H1} N. Hayashi, Local existence in time of small solutions to the Davey-Stewartson system, Annales de
 l'I.H.P.
 Physique Theorique \textbf{65} (1996), 313-366.

\bibitem[H2]{H2} N. Hayashi, Local existence in time of solutions to the elliptic-hyperbolic Davey-Stewartson system
 without smallness condition on the data, J.Analys\'e Math\'ematique \textbf{73} (1997), 133-164.

\bibitem[HH1]{HH1} N. Hayashi, H. Hirata, Global existence and asymptotic behaviour of small solutions to the elliptic-
hyperbolic Davey-Stewartson system, Nonlinearity \textbf{9} (1996),
1387-1409.

\bibitem[HH2]{HH2} N. Hayashi, H. Hirata, Local existence in time of small solutions to the elliptic-hyperbolic
 Davey-Stewartson system in the usual Sobolev space, Proc. Edinburgh Math. Soc. \textbf{40} (1997), 563-581.

\bibitem[HS]{HS} N. Hayashi, J.C. Saut, Global existence of small solutions to the Davey-Stewartson and the Ishimori
 systems,
Differential Integral Equations, \textbf{8} (1995), 1657-1675.

\bibitem[Ka]{Ka} V.L. Karpman, Stabilization of soliton instabilities by high-order dispersion: Fourth order nonlinear Schrödinger-type equations, Phys. Rev. E \textbf{53} (1996), 1336-1339.

\bibitem[KS]{KS} V.L. Karpman, A.G. Shagalov, Stabilitiy of soliton described by nonlinear Schrödinger-type equations with high-order dispersion, Phys. D \textbf{144} (2000), 194-210.

\bibitem[K1]{K1} T. Kato, On nonlinear
Schr\"odinger equations, II, $H^s$ solutions and unconditional well-posedness, J. Anal. Math. \textbf{67} (1995), 281-306.



 \bibitem[Ku]{Ku} E.A. Kuznetsov, Wave collapse in Nonlinear Optics, Topics in Applied Physics 114 (2009), Springer, 175-190.

\bibitem[LMSLP]{LMSLP} T. Lahaye, C. Menotti, L. Santos, M. Lewenstein, T. Pfau,
The physics of dipolar bosonic quantum gases, Rep. Prog. Phys. 72 (2009), 126401.


 \bibitem[LP]{LP} F. Linares, G. Ponce, Introduction to Nonlinear Dispersive
 Equations, Universitext, Springer, New York, 2009.

\bibitem[MXZ-1]{MXZ-1} C. Miao, G. Xu, L. Zhao, The dynamics of the 3D radial NLS with the combined terms, Comm. Math. Phys. 318 (3) (2013), 767-808.

\bibitem[MXZ-2]{MXZ-2} C. Miao, G. Xu, L. Zhao, The dynamics of the NLS with the combined terms in five and higher dimensions, in: Some Topics in Harmonic Analysis and Applications, in: Advanced Lectures in Mathematics, vol.34, Higher Education Press/International Press, Beijing/USA, 2015, pp.265-298.




\bibitem[O]{O} R. O`neil, Convolution Operators and $L^{p,q}$ Spaces, Duke Math. J. \textbf{30} (1963), 129-142.






\bibitem [RY1]{RY1} F. Ribaud, A. Youssfi, Regular and self-similar solutions of nonlinear Schr\"odinger equations,
J. Math Pures Appl. \textbf{77} (1998), 1065-1079.


\bibitem[Sh]{Sh} V.I. Shrira, On the propagation of a three-dimensional packet of weakly nonlinear internal gravity wave,
Int. J. Nonlinear Mech. \textbf{16} (1991), 129-138.


\bibitem[SFR]{SFR} P. Braz-Silva, L.C.F. Ferreira, E.J. Villamizar-Roa, On the existence of infinite energy solutions for nonlinear
 Schr\"odinger equations, Proc. Amer. Math. Soc. \textbf{137} (2009), 1977-1987.


\bibitem[TVZ]{TVZ} T. Tao, M. Visan, X. Zhang, The nonlinear Schr\"odinger equation with
combined power-type nonlinearities, Commun. Partial Differential Equations \textbf{32} (2007), 1281-1343.

\bibitem[Tr]{Tr} H. Triebel, Interpolation Theory, Function Spaces, Differential Operator, North-Holland Publishing Company, Berlin, 1978.

\bibitem[VP]{VP} E.J. Villamizar-Roa, J.E. P\'erez-L\'opez,  On the Davey-Stewartson system with singular initial data,  C. R. Math. Acad. Sci. Paris \textbf{350} (2012), 959--964.



\bibitem[W]{W} F.B. Weissler, Asymptotically self-similar solutions of the two power nonlinear Schr\"odinger equation,
 Adv. Differential Equations \textbf{6} no.4 (2001), 419-440.


\end{thebibliography}
\end{document}